\documentclass{amsart}
\usepackage{amsmath}

\usepackage{mathrsfs}
\usepackage{cases}
\usepackage{txfonts}
\usepackage{amsfonts}
\usepackage{mathrsfs}
\usepackage{amssymb}
\usepackage{amsmath,amssymb}
\usepackage{color}

\newtheorem{theorem}{Theorem}[section]
\newtheorem{lemma}[theorem]{Lemma}

\theoremstyle{definition}
\newtheorem{definition}[theorem]{Definition}

\theoremstyle{remark}
\newtheorem{remark}[theorem]{Remark}

\numberwithin{equation}{section} 

\begin{document}

\title[Wavelets, Multiplier spaces and Schr\"{o}dinger type operators ]
 {Wavelets, Multiplier spaces and application to Schr\"{o}dinger type operators with non-smooth potentials }

\author[Pengtao Li]{Pengtao Li}
\address{Department of Mathematics, Shantou University, Shantou, Guangdong, China.}
\email{ptli@stu.edu.cn}

\author[Qixiang Yang]{Qixiang Yang}
\address{School of Mathematics and Statics, Wuhan University, Wuhan, 430072, China.}
\email{qxyang@whu.edu.cn}
\author[Yueping Zhu]{Yueping Zhu}
\address{School of Scineces, Nantong University, Nantong, 226007, China.}
\email{zhuyueping@ntu.edu.cn}

\thanks{The research is supported by NSFC No. 11171203; New Teachers'
Fund for Doctor Stations, Ministry of Education 20114402120003; and
Guangdong Natural Science Foundation S2011040004131.}

\maketitle

\date{}

\begin{abstract}
 In this paper, we employ Meyer wavelets
to characterize multiplier spaces between Sobolev spaces without
using capacity.
 Further, we introduce logarithmic Morrey spaces $M^{t,\tau}_{r,p}(\mathbb{R}^{n})$
 to
establish the inclusion relation between Morrey spaces and multiplier
spaces. By wavelet characterization and fractal skills,
we construct a counterexample to show that the scope of the index
$\tau$ of $M^{t,\tau}_{r,p}(\mathbb{R}^{n})$ is sharp. As an application, we consider a Schr\"odinger type operator
with potentials in $M^{t,\tau}_{r,p}(\mathbb{R}^{n})$.
\end{abstract}

\section{Introduction}
 A
function $g$ is called a multiplier from $H^{t+r,p}(\mathbb{R}^{n})$
to $H^{t,p}(\mathbb{R}^{n})$ if for every function $f\in
H^{t+r,p}(\mathbb{R}^{n})$, the product $fg\in
H^{t,p}(\mathbb{R}^{n})$. We denote by $X^{t}_{r,p}(\mathbb{R}^{n})$
the class of all such functions $g$. As useful tools, multipliers on
the spaces of differential functions are applied to the study of
various problems in harmonic analysis and differential equations.
For example, the coefficients of a differential operator can be seen
as multipliers. For a function $u$ belonging to some Banach space,
M. Cannone reminded us that the non linear term $u^{2}$ can be
regarded as the product of a function $u$ in this Banach space and a
multiplier $u$. M. Cannone made many contributions on nonlinear
problems. See \cite{Ca1, Ca2, CK, Le}. For more information on both
multiplier spaces and PDE, we refer the readers to V. Maz'ya and T.
Shaposhnikova's celebrated monograph \cite{MS} and their recent
work.

In \cite{MS}, V. Maz'ya
and T. Shaposhnikova gave many characterizations of different kinds
of multiplier spaces. For example, they obtained that for $t\geq 0, r>0$, $1<p<{n}/{(r+t)}$,
the multiplier spaces $X^{t}_{r,p}(\mathbb{R}^{n})$ can be characterized  by capacity on
arbitrary compact sets.
The multiplier spaces
$X^{t}_{r,p}(\mathbb{R}^{n})$ are defined as follows.
\begin{definition}\label{def1}(\cite{MS})
Fix $1<p<\infty$ and $r, t\geq0$. The multiplier space
$X^{t}_{r,p}(\mathbb{R}^{n})$ is defined as the set of all the
functions $f(x)$ such that
$$\|f\|_{X^{t}_{r,p}(\mathbb{R}^{n})}:=\sup_{\|g\|_{H^{t+r,p}(\mathbb{R}^{n})}\leq1}\|fg\|_{H^{t,p}(\mathbb{R}^{n})}<\infty.$$
\end{definition}
For a compact set $e\subset\mathbb{R}^{n}$, the capacity ${\rm cap}(e,
H^{t,p})$ on $e$ is defined by
$${\rm cap}(e, H^{t,p})=\inf\Big\{\|u\|^{p}_{H^{t,p}(\mathbb{R}^{n})}:\ u\in \mathcal{S}, u\geq1 \text{ on }e\Big\},$$
where $\mathcal{S}$ is the Schwartz class of rapidly decreasing
smooth functions on $\mathbb{R}^{n}$.
\begin{lemma}\label{le1}{\rm(\cite{MS})}
 Given $r>0$ and $t\geq 0$.

(i) For $1<p<{n}/{(r+t)}$, $f\in X^{t}_{r,p}(\mathbb{R}^{n})$ if and only if
$$\sup_{e\subset\mathbb{R}^{n}}\left(\frac{\|(-\Delta)^{t/2}f\|_{L^{p}(e)}}{({\rm cap }(e, H^{t+r,p}))^{1/p}}
+\frac{\|f\|_{L^{p}(e)}}{({\rm cap }(e,
H^{r,p}))^{1/p}}\right)<\infty.$$

(ii) For $1<p<{n}/{r}$ and any cube $ Q$ with length less than 1, the capacity
${\rm{ cap}}(Q, H^{r,p})$ is less than $C|Q|^{1-{pr}/{n}}$.

\end{lemma}
Our motivation is based upon the following consideration. For
complicated compact sets, it is  very difficult to compute the
capacity. The main aim of this paper is to give some wavelet characterizations
and introduce  some  sufficient conditions which can be verified easily.
Precisely, for $r>0, t\geq0$ and
$t+r<1<p<{n}/{(r+t)}$, we will give a characterization of the
multipliers from $H^{t+r, p}(\mathbb{R}^{n})$ and $H^{r,
p}(\mathbb{R}^{n})$ by Meyer wavelets without using capacity.
See Theorems \ref{th3}. Also our method can be
applied to study the relation between multiplier spaces and
Morrey spaces. To deal with the case $t>0$, we have to introduce
the almost local operator $T^{t}$. See \S 2.

Lemma \ref{le1} implies that the multiplier space
$X^{t}_{r,p}(\mathbb{R}^{n})$ is a subspace of Morrey space
$M^{t}_{r,p}(\mathbb{R}^{n})$. It is natural to ask if the reverse
inclusion relation holds. Unfortunately, for $t=0$, the imbedding
$X^{t}_{r,p}(\mathbb{R}^{n})\subset M^{t}_{r,p}(\mathbb{R}^{n})$ is
not an isomorphism.  In \cite{Le}, P. G. Lemari\'e  gave a
counterexample to show that when $n-2r$ is an integer,
$X^{0}_{r,2}(\mathbb{R}^{n})\neq M^{0}_{r,2}(\mathbb{R}^{n})$.
Recently, P. G. Lemari\'e \cite{Le1} and Yang-Zhu \cite{YZ}
constructed some counterexamples for $t=0$ and $1<p<{n}/{r}$.

For $t> 0$, we have to consider the action of the differentiation,
so we can not construct counterexample like the case $t=0$ in \cite{YZ}.
Our counterexample  in Theorem \ref{th6} depends on our wavelet characterization, Theorem \ref{th10}
and fractal skills. From this counterexample, we can see that the product of $f\in
M^{t}_{r,p}(\mathbb{R}^{n})$ and $g\in H^{t+r, p}(\mathbb{R}^{n})$
may produce a blow up phenomenon of logarithmic type on fractal sets
with Hausdorff dimension $n-p(r+t)$.
To eliminate this defect,  we
introduce a logarithmic type Morrey space
$M^{t,\tau}_{r,p}(\mathbb{R}^{n})$ and prove that for
$\tau>{1}/{p'}$,
$$M^{t,\tau}_{r,p}(\mathbb{R}^{n})\subset
X^{t}_{r,p}(\mathbb{R}^{n})\subset M^{t}_{r,p}(\mathbb{R}^{n}) =
M^{t,0}_{r,p}(\mathbb{R}^{n}),$$ where $r>0,\ t\geq0$ and
$1<p<{n}/{(r+t)}$. See \S 4.

It should be pointed out that, in the
above inclusion relation, the scope of $\tau$ is
$({1}/{p'},~\infty)$, where $p'$ is the conjugate number of $p$.
In \S 5, our counterexample implies that, for  $0<\tau\le {1}/{p'}$ , there exists some
function $f\in M^{t,\tau}_{r,p}(\mathbb{R}^{n})$, but $f\notin
X^{t}_{r,p}(\mathbb{R}^{n})$. See \S 5 for the details.
Theorems \ref{th10} and  \ref{th6} illustrate the
difference between Morrey spaces and multiplier spaces.

Another motivation ʮis that the results about multipliers on Sobolev  spaces can be applied to
the study on partial differential equations. For example, in
\cite{MV}, V. Maz'ya and I. E. Verbitsky considered the multipliers
from $H^{1,2}(\mathbb{R}^{n})$ to
 $H^{-1,2}(\mathbb{R}^{n})$. For a Schr\"odinger operator $L=I-\Delta+V$,
they got many sufficient and necessary conditions such that $V$ is a
multiplier from ${H}^{1,2}(\mathbb{R}^{n})$ to
${H}^{-1,2}(\mathbb{R}^{n})$. For more information, we refer the
readers to \cite{Le, MS, MV, MV2} and the references therein.

As an application of our results, we consider the solution in Sobolev spaces $H^{t+r,p}(\mathbb{R}^{n})$ to the equation:
\begin{eqnarray}\label{eq7}
(I+(-\Delta)^{r/2}+V)f&=&g,
\end{eqnarray}
where $g\in H^{t,p}(\mathbb{R}^{n})$ and $V\in
M^{t,\tau}_{r,p}(\mathbb{R}^{n})$ with $r>0, t\geq0$,
$1<p<{n}/{(r+t)}$, $\tau>{1}/{p'}$. If $V$ is a function of
H\"{o}lder class, one usual method to deal with equation (\ref{eq7})
is the boundedness of Calder\'{o}n-Zygmund operators. As a function
in the logarithmic Morrey spaces $M^{t,\tau}_{r,p}(\mathbb{R}^{n})$,
$V$ may be not a $L^{\infty}$ function. In \S 6, by Theorem
\ref{th5}, we prove that for $V(x)\in
M^{t,\tau}_{r,p}(\mathbb{R}^{n})$, the above equation (\ref{eq7})
has an unique solution in the Sobolev space
$H^{t+r,p}(\mathbb{R}^{n})$.

In this paper, we use four tools in analysis. One is the
multi-resolution analysis introduced by Y. Meyer and S. Mallat in
1990s. The other is the almost local operator $T^{t}$. See \S 2. By
the projection operators generated from multi-resolution analysis,
S. Dobynski (cf. \cite{CLMS} ) obtained a decomposition of the
product of two functions. In order to adapt to our needs, we make
some modification to Dobynski's decomposition. The third main skills are
to use combination atoms and to introduce some differential methods.
The forth main skill is to choose special functions such that their
wavelet coefficients restricted on some fractal sets. See \S 4 \, and \S 5.

Our paper is organized as follows.
In \S 2, we state some notations and known results
which will be used throughout this paper. In \S 3, we give a
wavelet characterization of the multiplier spaces
$X^{t}_{r,p}(\mathbb{R}^{n})$.  In \S 4, we introduce a class
of logarithmic Morrey spaces $M^{t,\tau}_{r,p}(\mathbb{R}^{n})$ and
get a very simple sufficient condition of
$X^{t}_{r,p}(\mathbb{R}^{n})$. In \S 5, for $M^{t,\tau}_{r,p}(\mathbb{R}^{n})$, we
construct a counterexample to prove the sharpness of the scope of the index $\tau$ obtained in \S 4.
In the last section, we consider an application to PDE problem.

\section{Some preliminaries}
In this section, we state some notations, knowledge and preliminary
lemmas which will be used in the sequel. Firstly, we recall some
background knowledge of wavelets and multi-resolution analysis.

We will adopt real-valued tensor product wavelets  to study the
multiplier spaces in this paper. Let
$E_{n}=\{0,1\}^{n}\backslash\{0\}$. For $\varepsilon=0$
(respectively, $\varepsilon\in E_{n}$), we assume that
$\Phi^{\varepsilon}(x)$ is a scaling function (respectively,
wavelet). In the proof, we use only Meyer wavelets and regular
Daubechies wavelets. We say a Daubechies wavelet is regular if it
has sufficient vanishing moment until order $m$ and
$\Phi^{\varepsilon}(x)\in C^{m}_{0}([-2^{M}, 2^{M}])$, where the
regularity exponent $m$ is large enough and $M$ is determined by
$m$, see \cite{M, W} for more details. For any $\varepsilon\in \{0,
1\}^{n}, j\in\mathbb{N}$ and $k\in\mathbb{Z}^{n}$, we denote
$\Phi^{\varepsilon}_{j, k}(x)=2^{jn/2}\Phi^{\varepsilon}(2^{j}x-k)$.
In addition we define
$$\Lambda_{n}=\Big\{(\varepsilon, j, k): \varepsilon\in \{0, 1\}^{n},
j\in\mathbb{N}, k\in\mathbb{Z}^{n} \text{ and } \varepsilon\neq0,
\text{ if } j>0\Big\}.$$
For fixed tempered distribution $f$,
if we use wavelets which is sufficient regular,
then  we can define  $f^{\varepsilon}_{j, k}=\left<f,
\Phi^{\varepsilon}_{j, k}\right>$.
And the wavelet representation
$f =\sum\limits_{(\varepsilon, j, k)\in\Lambda_{n}}f^{\varepsilon}_{j, k}\Phi^{\varepsilon}_{j, k}$
holds in the sense of distribution.

Let $\left\{V^{1}_{j},\ j\in\mathbb{Z}\right\}$ be an orthogonal
multi-resolution in $L^{2}(\mathbb{R})$ with the scaling function
$\Phi^{0}(x)$. Denote by $W^{1}_{j}$ the orthogonal complement space
of $V^{1}_{j}$ in $V^{1}_{j+1}$, that is, $W^{1}_{j}= V^{1}_{j+1}
\ominus V^{1}_{j}$. Let $\left\{\Phi^{1}(x-k),\
k\in\mathbb{Z}\right\}$ be an orthogonal basis in $W^{1}_{0}$. For
$\varepsilon=(\varepsilon_{1},\cdots, \varepsilon_{n})\in
\{0,1\}^{n}$, denote $\Phi^{\varepsilon}(x)=
\prod^{n}\limits_{i=1}\Phi^{\varepsilon_{i}}(x_{i})$. For
$V_{j}=\left\{ f(x)= \sum\limits_{k\in \mathbb{Z}^{n}} f^{0}_{j,k}
\Phi^{0}_{j,k}(x), \mbox { where } \{f^{0}_{j,k}\}_{k\in
\mathbb{Z}^{n}}\in l^{2}\right\}$ and  $W_{j}=\left\{ f(x)=
\sum\limits_{\varepsilon\in E_{n}, k\in \mathbb{Z}^{n}}
f^{\varepsilon}_{j,k} \Phi^{\varepsilon}_{j,k}(x), \mbox { where }
\{f^{\varepsilon}_{j,k}\}_{\varepsilon\in E_{n}, k\in
\mathbb{Z}^{n}}\in l^{2}\right\}$, we have
\begin{lemma}\label{th9}
$\{V_{j},\ j\in\mathbb{Z}\}$ is an orthogonal multi-resolution with
the scaling function $\Phi^{0}(x)$.  $W_{j}$ is the orthogonal
complement space of $V_{j}$ in $V_{j+1}$, that is, $W_{j}= V_{j+1}
\ominus  V_{j}$. Further $\left\{\Phi^{\varepsilon}_{j,k},
(\varepsilon, j,k)\in \Lambda_{n} \right\}$ is  an orthogonal basis
in $V_{0}\bigoplus\limits_{j\geq 0} W_{j}=L^{2}(\mathbb{R}^{n})$.
\end{lemma}
Denote by $P_{j}$ and $Q_{j}$ the projection operators from
$L^{2}(\mathbb{R}^{n})$ to $V_{j}$ and $W_{j}$, respectively. By Lemma \ref{th9},
S. Dobynski got a decomposition of the product of two
functions $f$ and $g$, which is similar to Bony's paraproduct (see \cite{CLMS}).
Denote
$$\tilde{\Lambda}_{n}=\Big\{(\varepsilon,\varepsilon', j,k,k'),\
\varepsilon, \varepsilon'\in\{0,1\}^{n}\backslash\{0\}, j\geq0,
k,k'\in\mathbb{Z}^{n}, (\varepsilon,k)\neq(\varepsilon',k')\Big\}.$$
By the projection operators $P_{j}$ and $Q_{j}$, we divide the
product of $f(x)$ and $g(x)$ into the following terms.
\begin{eqnarray*} f(x)
g(x)&=&P_{0}(f)P_{0}(g)+\sum_{j\geq0}P_{j}(f)Q_{j}(g)+\sum_{j\geq0}Q_{j}(f)P_{j}(g)\\
&+&\sum_{\tilde{\Lambda}_{n}}f^{\varepsilon}_{j,k}g^{\varepsilon'}_{j,k'}
\Phi^{\varepsilon}_{j,k}(x)\Phi^{\varepsilon'}_{j,k'}(x)+\sum_{\Lambda_{n},
\varepsilon\neq0}f^{\varepsilon}_{j,k}g^{\varepsilon}_{j,k}\left(\Phi^{\varepsilon}_{j,k}(x)\right)^{2}.
\end{eqnarray*}
To facilitate our use, we make a modification to the above
decomposition and use special wavelets for different cases. Let $N$
be a positive integer. We decompose the product $f(x)g(x)$ as
\begin{equation}\label{for1}
\begin{split}
fg&=\sum^{\infty}_{j=1}\left[P_{j+1}(f)P_{j+1}(g)-P_{j}(f)P_{j}(g)\right]+P_{0}(f)P_{0}(g)\\
&=\sum^{\infty}_{j=N}\left[Q_{j}(f)Q_{j}(g)+P_{j}(f)Q_{j}(g)+Q_{j}(f)P_{j}(g)\right]+P_{N}(f)P_{N}(g)
\end{split}
\end{equation}
and the term $\sum^{\infty}\limits_{j=N}Q_{j}(f)P_{j}(g)$ can be decomposed
as

\begin{equation}\label{for2}
\begin{split}
\sum^{\infty}_{j=N}Q_{j}(f)P_{j}(g)
&=\sum^{\infty}_{j=N}Q_{j}(f)\sum^{N}_{t=1}Q_{j-t}(g)+\sum^{\infty}_{j=N}Q_{j}(f)P_{j-N}(g)\\
&=\sum^{\infty}_{j=0}\sum^{N}_{t=1}Q_{j+t}(f)Q_{j}(g)+\sum^{\infty}_{j=0}Q_{j+N}(f)P_{j}(g).
\end{split}
\end{equation}

For any
$j\in\mathbb{N}$ and $k=(k_{1}, k_{2},\dots,
k_{n})\in\mathbb{Z}^{n}$, let $Q_{j,\
k}=\prod\limits^{n}_{s=1}[2^{-j}k_{s}, 2^{-j}(k_{s}+1)]$ and denote
by $\Omega$ the set of all dyadic cubes $Q_{j,\ k}$.
For arbitrary set $Q$, we denote by $\tilde{Q}$ the
$2^{M+2}-$multiple of $Q$. Finally, let $\chi(x)$ be the
characteristic function of the unit cube $Q_{0}$ and $\tilde{\chi}$ be the
characteristic function of $\tilde{Q}_{0}$.

In 1970s, H. Triebel introduced Triebel-Lizorkin spaces
$F^{r,q}_{p}(\mathbb{R}^{n})$ (\cite{T}). Many function spaces can
be seen as the special cases for $F^{r,q}_{p}(\mathbb{R}^{n})$. For
example, $F^{r,2}_{1}(\mathbb{R}^{n})$ is the fractional Hardy
space. For $1<p<\infty$, $F^{r,2}_{p}(\mathbb{R}^{n})$ are the
Sobolev spaces $H^{r,p}(\mathbb{R}^{n})$. For $p=\infty$,
$F^{-r,2}_{\infty}(\mathbb{R}^{n})$ is the fractional BMO space
$BMO^{r}(\mathbb{R}^{n})$ which is defined by
$BMO^{r}(\mathbb{R}^{n}):=(I-\Delta)^{-{r}/{2}}BMO(\mathbb{R}^{n})$,
where $I$ is the unit operator, $\Delta$ is the Laplace operator.
Here $BMO(\mathbb{R}^{n})$ denotes the non-homogeneous bounded mean
oscillation space which is defined as the set of the functions such
that
$$\sup\limits_{|Q|=1}|f_{Q}|=\sup\limits_{Q}\frac{1}{|Q|}\left|\int_{Q}f(x)dx\right|\leq
C$$ and
$$\sup_{|Q|\leq1}\frac{1}{|Q|}\int_{Q}|f(x)-f_{Q}|^{2}dx<\infty.$$

For $1\leq p<\infty$ and $r\in\mathbb{R}$, it is well known that
$\left(F^{r,2}_{p}(\mathbb{R}^{n})\right)'=F^{-r,
2}_{p'}(\mathbb{R}^{n})$. The following lemma gives a
characterization of $F^{r,2}_{p}(\mathbb{R}^{n})$ by Meyer wavelets
and regular Daubechies wavelets. For the proof, we refer the readers
to \cite{Y1, YSY}.
\begin{lemma}\label{le2}
(i) For $1\leq p<\infty$ and $|r|<m$, using Meyer wavelets or
$m$-regular Daubechies wavelets,  we have the following equivalent
characterizations,
\begin{eqnarray*}
&&g(x)=\sum_{(\varepsilon, j, k)\in\Lambda_{n}}g^{\varepsilon}_{j,
k}\Phi^{\varepsilon}_{j,k}(x)\in F^{r,2}_{p}(\mathbb{R}^{n})\\
&\Longleftrightarrow&\left\|\left(\sum_{(\varepsilon, j,k)\in
\Lambda_{n}}2^{2j(r+{n}/{2})}|g^{\varepsilon}_{j,
k}|^{2}\chi(2^{j}-k)\right)^{1/2}\right\|_{L^{p}}<\infty\\
&\Longleftrightarrow&\left\|\left(\sum_{(\varepsilon, j,k)\in
\Lambda_{n}}2^{2j(r+{n}/{2})}|g^{\varepsilon}_{j,
k}|^{2}\tilde{\chi}(2^{j}-k)\right)^{1/2}\right\|_{L^{p}}<\infty.
\end{eqnarray*}
(ii) Given $|r|<m$. $g(x)=\sum\limits_{(\varepsilon,
j,k)\in\Lambda_{n}}g^{\varepsilon}_{j, k}\Phi^{\varepsilon}_{j,
k}(x)\in F^{r,2}_{\infty}(\mathbb{R}^{n})$ if and only if there
exists $1<p<\infty$ such that for any $Q\in\Omega$,
$$\left\|\left(\sum_{\varepsilon\in E_{n}, Q_{j,k}\subset Q}2^{2j(r+{n}/{2})}
|g^{\varepsilon}_{j,k}|^{2}\chi(2^{j}-k)\right)^{1/2}\right\|_{L^{p}}\leq
C|Q|^{1/p}.$$
\end{lemma}
The wavelet characterizations of function spaces have been studied by many authors. In
Chapters 5 and 6 of \cite{M}, Y. Meyer established wavelet
characterizations for many function spaces. In \cite{YCP}, Q. Yang, Z. Cheng
and L. Peng considered wavelet characterization of Lorentz type
Triebel-Lizorkin spaces and Lorentz type Besov spaces.
In \cite{Y1}, Q. Yang introduced the wavelet
definition of Besov type Morrey  spaces. W. Yuan, W. Sickel
and D. Yang considered the atomic decomposition for
Besov type Morrey spaces and Triebel-Lizorkin type Morrey  spaces in \cite{YSY}.

Morrey spaces $M^{t}_{r,p}(\mathbb{R}^{n})$ were introduced by
Morrey in 1938 and play an important role in the research of partial
differential equations. In 2003,  Wu and Xie \cite{WX} proved that
generalized Morrey spaces are also generalization of $Q$-type
spaces. In recent 20 years, $Q$-type spaces are studied extensively
(cf \cite{EJPX, PY, Y1,  YSY}).

Let $f_{t, Q}$ be the mean value of $(I-\Delta)^{{t}/{2}}f$ on a cube $Q$,
 $$f_{t, Q}=\frac{1}{|Q|}\int_{Q}(I-\Delta)^{{t}/{2}}f(x)dx.$$
The Morrey spaces $M^{t}_{r,p}(\mathbb{R}^{n})$ are defined as
follows.
\begin{definition}
For $1\leq p<\infty$ and $r$, $t\geq0$, the Morrey space
$M^{t}_{r,p}(\mathbb{R}^{n})$ is defined as the set of the functions
$f$ such that $\sup\limits_{|Q|=1}|f_{t,Q}|\leq C$ and
$$\int_{Q}\left|(I-\Delta)^{{t}/{2}}f(x)-f_{t, Q}\right|^{p}dx\leq C|Q|^{1-p(r+t)/n},$$
where $Q$ is any cube in $\mathbb{R}^{n}$ with $|Q|\leq1$.
\end{definition}
  In \cite{PY, YSY}, the authors proved that Morrey spaces
$M^{t}_{r,p}(\mathbb{R}^{n})$ can  be also characterized by wavelets.
We state it as the
following theorem and refer to \cite{YSY} for the proof.
\begin{theorem}\label{th7}
Given $t\in\mathbb{R}$, $1<p<\infty$ and $0\leq p(r+t)<n$.
$$f(x)=\sum_{(\varepsilon,
j,k)\in\Lambda_{n}}f^{\varepsilon}_{j,k}\Phi^{\varepsilon}_{j,k}(x)\in
M^{t}_{r,p}(\mathbb{R}^{n})$$ if and only if for any $ Q\in\Omega$
with $|Q|\leq1$,
$$\int_{Q}\left(\sum_{\varepsilon\in E_{n}, Q_{j,k}\subset Q}
2^{j(n+2t)}|f^{\varepsilon}_{j,k}|^{2}\chi(2^{j}x-k)\right)^{p/2}dx\leq
C|Q|^{1-{p(r+t)}/{n}}.$$
\end{theorem}
By Lemmas \ref{le1} and \ref{le2}, the
multiplier spaces $X^{t}_{r,p}(\mathbb{R}^{n})$ are also subspaces of
Morrrey spaces $M^{t}_{r,p}(\mathbb{R}^{n})$.
\begin{lemma}\label{le3}
Given $r>0$, $t\geq0$ and $1<p<{n}/{(r+t)}$. If $f\in
X^{t}_{r,p}(\mathbb{R}^{n})$, then $f(x)\in
M^{t}_{r,p}(\mathbb{R}^{n})$.
\end{lemma}

 Now we give two
lemmas about the fractional BMO spaces  $BMO^{r}(\mathbb{R}^{n})$. In the first lemma, we prove that
Morrey spaces $M^{t}_{r, p}(\mathbb{R}^{n})$ are subspaces of
$BMO^{r}(\mathbb{R}^{n})$.
\begin{lemma}\label{le4}
For $r>0,\ t\geq0$ and $1<p<{n}/{(r+t)}$,
$M^{t}_{r,p}(\mathbb{R}^{n})\subset BMO^{r}(\mathbb{R}^{n})$.
\end{lemma}
\begin{proof}
For any dyadic cube $Q$, we have
\begin{eqnarray*}
&&\int\left(\sum_{\varepsilon\in E_{n}, Q_{j,k}\subset
Q}~2^{jn-2jr}|f^{\varepsilon}_{j,k}|^{2}\chi(2^{j}x-k)\right)^{p/2}dx\\
&\leq&|Q|^{{p(r+t)}/{n}}\int\left(\sum_{\varepsilon\in E_{n},
Q_{j,k}\subset
Q}~2^{j(n+2t)}|f^{\varepsilon}_{j,k}|^{2}\chi(2^{j}x-k)\right)^{p/2}dx\\
&\leq&C|Q|^{{p(r+t)}/{n}}|Q|^{1-{p(r+t)}/{n}}\\
&\leq&C|Q|.
\end{eqnarray*}
\end{proof}
\begin{lemma}\label{le5}
Suppose $r>0$ and $\
f=\sum\limits_{(\varepsilon,j,k)\in\Lambda_{n}}f^{\varepsilon}_{j,k}\Phi^{\varepsilon}_{j,k}(x)\in
BMO^{r}(\mathbb{R}^{n})$. The wavelet coefficients of $f$
satisfy
$$|f^{\varepsilon}_{j,~k}|\leq C2^{(r-{n}/{2})j},\ \forall
\varepsilon\in\{0,1\}^{n},\ j\in \mathbb{N},\ k\in\mathbb{Z}^{n}.$$
\end{lemma}
\begin{proof}
Take $j\in\mathbb{N}$ and $k\in\mathbb{Z}^{n}$. We consider two
cases $\varepsilon\in E_{n}$ and $\varepsilon=0$ separately.

(i) For any $\varepsilon\in E_{n}$, by the definition of
$BMO^{r}(\mathbb{R}^{n})$, we get
$$\int\left(2^{jn-2jr}|f^{\varepsilon}_{j,k}|^{2}\chi(2^{j}x-k)\right)^{p/2}dx\leq C2^{-jn}.$$
It is easy to see that $|f^{\varepsilon}_{j,k}|\leq
C2^{j(r-{n}/{2})}$.

(ii) For $\varepsilon=0$,
$$f^{0}_{j,k}=\left<\sum_{(\varepsilon', j', k')\in\Lambda_{n}}f^{\varepsilon'}_{j', k'}\Phi^{\varepsilon'}_{j',k'},\Phi^{0}_{j,k}\right>
=\left<\sum_{j'<j}f^{\varepsilon'}_{j',
k'}\Phi^{\varepsilon'}_{j',k'},\Phi^{0}_{j,k}\right>.$$ Because
$\left|\sum\limits_{j'<j}f^{\varepsilon'}_{j',
k'}\Phi^{\varepsilon'}_{j',k'}(x)\right|\leq C2^{rj}$, we have
$|f^{0}_{j,k}|\leq C\left<2^{rj}, |\Phi^{0}_{j,k}(x)|\right>\leq
C2^{j(r-{n}/{2})}.$
\end{proof}

 Let $\Psi^{1}$ and $\Psi^{2}$ be two functions
in $C^{\mu}_{0}([-2^{M+1}, 2^{M+1}]^{n})$ with vanishing moments
$\int x^{\alpha}\Psi^{i}(x)dx=0$,  where $|\alpha|\leq \mu$ and
$i=1,2.$
Denote
$$a_{j,k, j',k'}=\left<\Psi^{1}_{j,k},\ \Psi^{2}_{j', k'}\right>.$$
The following lemma can be found in Chapter 8 of \cite{M} or Chapter
6 of \cite{Y1}.
\begin{lemma}\label{le6}
Given $|\mu|\leq m$. For $|s|<\mu$, the coefficients $a_{j,k,j',k'}$
satisfy the following condition:
\begin{equation}\label{eq8}
|a_{j,k,j',k'}|\leq C2^{-|j-j'|({n}/{2}+s)}\left(\frac{2^{-j}+2^{-j'}}{2^{-j}+2^{-j'}+|k2^{-j}-k'2^{-j'}|}\right)^{n+s}.
\end{equation}
\end{lemma}
By wavelet characterization of $H^{r, p}(\mathbb{R}^{n})$, the
continuity of Calder\'{o}n-Zygmund operators on
$H^{r,p}(\mathbb{R}^{n})$ is equivalent to the following lemma.  We
refer the readers to  \cite{M, MY, Y1} for the proof.
\begin{lemma}\label{le7}
Suppose $s>|r|$ and $g(x)=\sum\limits_{(\varepsilon,
j,k)\in\Lambda_{n}}g^{\varepsilon}_{j,k}\Phi^{\varepsilon}_{j,k}(x)\in
H^{r, p}(\mathbb{R}^{n})$. Let
$\tilde{g}^{\varepsilon}_{j,k}=\sum\limits_{(\varepsilon, j,
k)\in\Lambda_{n}}a^{\varepsilon,\varepsilon'}_{j,k,j',k'}g^{\varepsilon}_{j',k'}$.
If the coefficients $a^{\varepsilon,\varepsilon'}_{j,k,j',k'}$
satisfy the condition (\ref{eq8}), then we have
\begin{eqnarray*}
&&\int\left(\sum_{(\varepsilon,j,k)\in\Lambda_{n}}2^{j(n+2r)}|\tilde{g}^{\varepsilon}_{j,k}|^{2}\chi(2^{j}-k)\right)^{p/2}dx
\leq
C\int\left(\sum_{(\varepsilon,j,k)\in\Lambda_{n}}2^{j(n+2r)}|g^{\varepsilon}_{j,k}|^{2}\chi(2^{j}-k)\right)^{p/2}dx.
\end{eqnarray*}
\end{lemma}
We say that $T$ is a local operator if there exists some constant
$C>1$ such that for all $x\in \mathbb{R}^{n}$ and $r>0$, $T$ maps a
distribution with the support $B(x,r)$ to another distribution
supported on the ball $B(x, Cr)$. If ${t}/{2}$ is not a non-negative
integer,  the operator $(I-\Delta)^{t/2}$ is not a local operator.
Now we use wavelets to construct some special fractional
differential operators $T^{t}$, which are almost local operators and
will be used in the proof of our main result.
\begin{definition}\label{def8}
For $t\geq0$ and $h(x)=\sum\limits_{(\varepsilon, j,
k)\in\Lambda_{n}}h^{\varepsilon}_{j,k}\Phi^{\varepsilon}_{j,k}(x),$
 we define an operator $T^{t}$ corresponding to the kernel
$K^{t}(x,y)=\sum\limits_{(\varepsilon,j,k)\in\Lambda_{n}}2^{-jt}\Phi^{\varepsilon}_{j,k}(x)\Phi^{\varepsilon}_{j,k}(y)$
as
$$T^{t}h(x)=\sum\limits_{(\varepsilon,j,k)\in\Lambda_{n}}2^{-jt}h^{\varepsilon}_{j,k}\Phi^{\varepsilon}_{j,k}(x).$$
\end{definition}
It is easy to prove that $T^{0}$ is the identity operator and
$\|T^{t}h\|_{L^{p}}=\|h\|_{H^{-t,p}}$ for $1<p<\infty$. Furthermore,
we have
\begin{lemma}\label{le8}
Suppose $t\geq 0$. For any $Q_{j,k}\in \Omega$ and $x\in Q_{j,k}$,
 $2^{j({n}/{2}-t)}|h^{0}_{j,k}|\leq CMT^{t}h(x),$ where
$M$ is the Hardy-Littlewood maximal operator.
\end{lemma}
\begin{proof}
If $t=0$, the proof was given by Meyer \cite{M}. Now we consider the
case $t>0$. Since $T^{-t}\Phi^{0}(x)=\int K^{-t}(x,y)\Phi^{0}(y)dy$,
it is easy to verify that
$$|T^{-t}\Phi^{0}(x)|\leq C(1+|x|)^{-n-t}.$$

By the fact that $t>0$, we have
$$2^{j({n}/{2}-t)}h^{0}_{j,k}=2^{j({n}/{2}-t)}\left<T^{t}h(x),\ T^{-t}\Phi^{0}_{j,k}(x)\right>
=2^{{jn}/{2}}\left<T^{t}h(x),\ (T^{-t}\Phi^{0})_{j,k}(x)\right>.$$
Hence we can get
\begin{eqnarray*}
|2^{j({n}/{2}-t)}h^{0}_{j,k}|&=&2^{{jn}/{2}}\left|\left<T^{t}h(x),\
(T^{-t}\Phi^{0})_{j,k}(x)\right>\right|\\
&\leq&C2^{{jn}/{2}}\int
|T^{t}h(x)|2^{{jn}/{2}}\frac{dx}{(1+|2^{j}x-k|)^{n+t}}\\
&\leq&C2^{jn}\left(\int_{|2^{j}x-k|\leq1}|T^{t}h(x)|dx
+\sum_{l=1}^{\infty}\int_{2^{l-1}<\leq|2^{j}x-k|\leq2^{l}}|T^{t}h(x)|\frac{dx}{(1+|2^{j}x-k|)^{n+t}}\right)\\
&\leq&C2^{jn}\left(2^{-jn}M(T^{t}h)(x)+\sum^{\infty}_{l=1}2^{-lt}M(T^{t}h)(x)2^{-jn}\right)\\
&\leq&CM(T^{t})h(x).
\end{eqnarray*}
This completes the proof of Lemma \ref{le8}.
\end{proof}

In the rest of this section,  we give a decomposition of Sobolev
spaces associated with combination atoms. For $|r|<m$ and
$g(x)=\sum\limits_{(\varepsilon,j,k)\in\Lambda_{n}}g^{\varepsilon}_{j,k}\Phi^{\varepsilon}_{j,k}(x)$,
denote
$$S_{r}g(x)=\left(\sum_{(\varepsilon,j,k)\in\Lambda_{n}}2^{j(2r+n)}|g^{\varepsilon}_{j,k}|^{2}\chi(2^{j}x-k)\right)^{1/2}$$
and for $t=0$, denote also $Sg(x)=S_{0}g(x)$.
\begin{definition}\label{def4}
Given $r\in\mathbb{R},\ \lambda>0$. For arbitrary measurable set
$E$, we say that $g(x)$ is a $(r,\lambda,E)-$combination atom, if
$\text{ supp}(S_{r}g)\subset E$ and $S_{r}g(x)\leq\lambda$. If $E$
is a dyadic cube, then $g(x)$ is called a $(r,\lambda, E)-$atom.
\end{definition}
In \cite{Y2}, Q. Yang introduced the combination atom
decomposition of Lebesgue spaces. In this paper, we need a similar result for
Sobolev spaces.
\begin{theorem}\label{th2}
If $1<p<\infty$, $|r|<m$ and $\|g\|_{H^{r,p}}\leq 1$,  there exists
a series of $(r, 2^{v}, E_{v})$-combination atoms $g_{v}(x)$ such
that $\sum\limits_{v\in\mathbb{N}}2^{pv}|E_{v}|\leq C$.
\end{theorem}
\begin{proof}
Denote
$$\tilde{S}_{r}g(x)=\left(\sum_{(\varepsilon,j,k)\in\Lambda_{n}}2^{j(2r+n)}|g^{\varepsilon}_{j,k}|^{2}\tilde{\chi}(2^{j}-k)\right)^{1/2}.$$
For  $v\geq 1$, let $E_{v}=\left\{x:
\tilde{S}_{r}g(x)>2^{v}\right\}$. By wavelet characterization
of Sobolev spaces, we have
$\sum\limits_{v\in\mathbb{N}}2^{pv}|E_{v}|\leq C$. Let
$E_{v}=\bigcup\limits_{l}Q^{v,l}$, where $Q^{v,l}$ are disjoint
maximal dyadic cubes with $|Q^{v,l}|\leq 1$. Let
$\mathfrak{F}_{v,l}$ be the set of dyadic cubes contained in
$Q^{v,l}$ but not in $E_{v+1}$,
$\mathfrak{F}_{v}=\bigcup\limits_{l}\mathfrak{F}_{v,l}$ and
$\mathfrak{F}_{0}=\Omega\backslash\bigcup\limits_{v\geq1}\mathfrak{F}_{v}$.
Let $E_{0}=\{x\in Q, Q\in\mathfrak{F}_{0}\}$ and we can write also
$E_{0}=\bigcup\limits_{l}Q^{0,l}$, where $Q^{0,l}$ are disjoint
maximal dyadic cubes in $\Omega$. The  related set $\mathfrak{F}_{0,l}$ is defined as
$\mathfrak{F}_{0,l}=\left\{Q\subset Q^{0,l} \mbox
{ and }
Q\in \mathfrak{F}_{0}\right\}
$.

For any $v\geq 0$, we write
$g_{v,l}(x)=\sum\limits_{Q_{j,k}\in\mathfrak{F}_{v,l}}g^{\varepsilon}_{j,k}\Phi^{\varepsilon}_{j,k}(x)$
and
$g_{v}(x)=\sum\limits_{Q_{j,k}\in\mathfrak{F}_{v}}g^{\varepsilon}_{j,k}\Phi^{\varepsilon}_{j,k}(x)$.
Then $g_{v}(x)$ is a desired combination atom. This completes the proof.
\end{proof}

\section{Wavelet characterization of the multiplier spaces}

In this section, we use Meyer wavelets to characterize the
multiplier spaces $X^{t}_{r,p}(\mathbb{R}^{n})$. For any $ g\in
H^{t,p}(\mathbb{R}^{n})$, let
$g^{\Phi,\varepsilon}_{j,k}=\left<g(x),\
2^{jn}(\Phi^{\varepsilon})^{2}(2^{j}x-k)\right>$. Let $\Phi(x)$ be a
function satisfying $\Phi(x)\geq 0,\ \Phi(x)\in
C^{\infty}_{0}(B(0,1))$ and $\int\Phi(x)dx=1$. For any $ g\in
H^{t,p}(\mathbb{R}^{n})$, define $g^{\Phi}_{j,k}=\left< g(x),\
2^{jn}\Phi(2^{j}x-k)\right>$. The function spaces
$S^{t}_{r,p}(\mathbb{R}^{n})$ and $S^{\Phi, t}_{r,
p}(\mathbb{R}^{n})$ are defined as follows.
\begin{definition}\label{def5}
Given $r>0$, $t\geq 0$ and $r+t<1<p<{n}/{(r+t)}$.

(i) We say $f(x)\in S^{t}_{r,p}(\mathbb{R}^{n})$ if
$f(x)=\sum\limits_{(\varepsilon,j,k)\in \Lambda_{n}
}f^{\varepsilon}_{j,k}\Phi^{\varepsilon}_{j,k}(x)$ and
$$\int\left(\sum\limits_{(\varepsilon,j,k)\in\Lambda_{n}}2^{j(n+2t)}|g^{\Phi,\varepsilon}_{j,k}|^{2}|f^{\varepsilon}_{j,k}|^{2}
\chi(2^{j}x-k)\right)^{p/2}dx\leq C,$$ where $ g\in H^{t+r,
p}(\mathbb{R}^{n})$ and $ \|g\|_{H^{r+t,p}(\mathbb{R}^{n})}\leq1.$

(ii) We say $f(x)\in S^{\Phi, t}_{r,p}(\mathbb{R}^{n})$ if
$f(x)=\sum\limits_{(\varepsilon,j,k)\in
\Lambda_{n}}~f^{\varepsilon}_{j,k}\Phi^{\varepsilon}_{j,k}(x)$ and
$$\int\left(\sum_{(\varepsilon,j,k)\in\Lambda_{n}}2^{j(n+2t)}|g^{\Phi}_{j,k}|^{2}|f^{\varepsilon}_{j,k}|^{2}
\chi(2^{j}x-k)\right)^{p/2}dx\leq C,$$ where $g\in
H^{r+t,p}(\mathbb{R}^{n})$ and
$\|g\|_{H^{r+t,p}(\mathbb{R}^{n})}\leq 1$.
\end{definition}
Now we give a  wavelet characterization of the multiplier space
$X^{t}_{r,p}(\mathbb{R}^{n})$. Let $\Phi^{0}(x)$ and
$\Phi^{\varepsilon}(x),\ \varepsilon\in E_{n}$ be the scaling
function and wavelet functions, respectively. For $(\varepsilon, j,
k)$, $(\varepsilon', j', k')$, $(\varepsilon", j', k')\in
\Lambda_{n}$ and $l\in\mathbb{Z}^{n}$, let
$$a^{\varepsilon, \varepsilon'}_{j,k, l, j', k'}=\left<\Phi^{0}_{j,k+l}(x)\Phi^{\varepsilon}_{j,k}(x),\ \Phi^{\varepsilon'}_{j',k'}(x)\right>$$
and
$$a^{\varepsilon, \varepsilon, \varepsilon'', 0}_{j,k, 0, j',
k'} =\left<(\Phi^{\varepsilon}_{j,k})^{2}-2^{jn}\Phi(2^{j}x-k),\ \Phi^{\varepsilon''}_{j',k'}(x)\right>.$$

Furthermore, for $0\leq s\leq N$, $\varepsilon'\in E_{n}$, $l\in \mathbb{Z}^{n}$
and $s+|\varepsilon-\varepsilon'|+|l|\neq0$, let
$$a^{\varepsilon, \varepsilon', \varepsilon'', s}_{j,k, l, j',
k'} =\left<\Phi^{\varepsilon'}_{j,k}(x)\Phi^{\varepsilon}_{j+s,2^{s}k+l}(x),\
\Phi^{\varepsilon''}_{j',k'}(x)\right>.$$ The following lemma is
obtained in \cite{M}.
\begin{lemma}\label{le9}
There exist sufficient big integers $N$, $N_{1}$ and $N_{2}$ such
that $\min\left\{N, N_{1}, N_{2}\right\}>8n+8m$ and the following
estimates hold.

 (i) If $(\varepsilon, j,k)$, $(\varepsilon', j', k')\in\Lambda_{n}$, $l\in \mathbb{Z}^{n}$ and $j\geq j'$, then
$$|a^{\varepsilon,\varepsilon'}_{j,k,l,j',k'}|\leq C(1+|l|)^{-N_{1}}2^{{nj'}/{2}+j'-j}(1+|k'-2^{j'-j}k|)^{-N_{2}}.$$

(ii) If $(\varepsilon, j, k)$, $(\varepsilon'', j',
k')\in\Lambda_{n}$, $0\leq s\leq N$, $\varepsilon'\in E_{n}$,
$l\in\mathbb{Z}^{n}$ and $j\geq j'$, then
$$|a^{\varepsilon,\varepsilon',\varepsilon'',s}_{j,k,l,j',k'}|\leq C(1+|l|)^{-N_{1}}2^{{nj'}/{2}+j'-j}(1+|k'-2^{j'-j}k|)^{-N_{2}}.$$

(iii) If $(\varepsilon,j,k)$, $(\varepsilon', j',
k')\in\Lambda_{n}$, $l\in\mathbb{Z}^{n}$ and $j<j'$, then
$$|a^{\varepsilon,\varepsilon'}_{j,k,l,j',k'}|\leq C(1+|l|)^{-N_{1}}2^{-{nj'}/{2}+nj+N(j-j')}(1+|k-2^{j-j'}k'|)^{-N_{2}}.$$

(iv) If $(\varepsilon,j,k)$, $(\varepsilon'', j',k')\in\Lambda_{n}$,
$0\leq s\leq N$, $\varepsilon'\in E_{n}$, $l\in\mathbb{Z}^{n}$ and
$j<j'$, then
$$|a^{\varepsilon,\varepsilon', \varepsilon'',s}_{j,k,l,j',k'}|\leq C(1+|l|)^{-N_{1}}2^{-{nj'}/{2}+nj+N(j-j')}(1+|k-2^{j-j'}k'|)^{-N_{2}}.$$
\end{lemma}

\begin{theorem}\label{th3}
For $t\geq 0,r>0$ and $t+r<1<p<{n}/{(t+r)}$,   there exist two equivalent relations between
$X^{t}_{r,p}(\mathbb{R}^{n})$ and
$M^{t}_{r,p}(\mathbb{R}^{n})$.

(i) $f\in X^{t}_{r,p}(\mathbb{R}^{n})$ if and only if $f\in
M^{t}_{r,p}(\mathbb{R}^{n})$ and $f\in S^{t}_{r,p}(\mathbb{R}^{n})$.

(ii) $f\in X^{t}_{r,p}(\mathbb{R}^{n})$ if and only if $f\in
M^{t}_{r,p}(\mathbb{R}^{n})$ and $f\in S^{\Phi,
t}_{r,p}(\mathbb{R}^{n})$.
\end{theorem}
\begin{proof}
Let $\Phi^{0}$ and $\Phi^{\varepsilon}$ be the scaling function and
wavelet functions of Meyer wavelets, respectively. There exists an
integer $N\geq 3$ such that $\int
x^{\alpha}\Phi^{0}(x)\Phi^{\varepsilon}(2^{N}x-k)dx=0,\ \forall
k\in\mathbb{Z}^{n},\alpha \in \mathbb{N}^{n},\ \forall
\varepsilon\in E_{n}$. Denote by $\Lambda_{\varepsilon, n}$ the set
$$\left\{(s,\ \varepsilon',\ l),\ 0\leq s\leq N,\ \varepsilon'\in E_{n},\ l\in\mathbb{Z}^{n},\ |l|\leq 2^{(M+2+s)n}
\text{ and if }s=0, \text{ then }(0, \varepsilon, 0)\neq(0,
\varepsilon', 0)\right\}.$$ Let $h(x)$ be any function in
$H^{-t,p'}(\mathbb{R}^{n})$. We  prove that for $f\in
M^{t}_{r,p}(\mathbb{R}^{n})\cap S^{t}_{r,p}(\mathbb{R}^{n})$, $fh\in
H^{-t-r,p'}(\mathbb{R}^{n})$ and
$$\|fh\|_{H^{-t-r,p'}(\mathbb{R}^{n})}\leq
C\|f\|_{M^{t}_{r,p}(\mathbb{R}^{n})\cap
S^{t}_{r,p}(\mathbb{R}^{n})}\|h\|_{H^{-t,p'}(\mathbb{R}^{n})}.$$

In fact, if the above inequality holds,  we have
\begin{eqnarray*}
\|f\|_{X^{t}_{r,p}(\mathbb{R}^{n})}&=&\sup_{\|g\|_{H^{t+r,p}(\mathbb{R}^{n})}\leq
1}\|f g\|_{H^{t,p}(\mathbb{R}^{n})}\\
&=&\sup_{\|g\|_{H^{r+t,p}(\mathbb{R}^{n})}\leq1}\sup_{\|h\|_{H^{-t,p'}(\mathbb{R}^{n})}}|\left<fg,\ h\right>|\\
&=&\sup_{\|g\|_{H^{r+t,p}(\mathbb{R}^{n})}\leq1}\sup_{\|h\|_{H^{-t,p'}(\mathbb{R}^{n})}}|\left<fh,\ g\right>|\\
&\leq&\sup_{\|g\|_{H^{r+t,p}(\mathbb{R}^{n})}\leq1}
\sup_{\|h\|_{H^{-t,p'}(\mathbb{R}^{n})}}\|f h\|_{H^{-t-r,p'}(\mathbb{R}^{n})}\|g\|_{H^{t+r,p}(\mathbb{R}^{n})}\\
&\leq&\|f\|_{M^{t}_{r,p}(\mathbb{R}^{n})\cap
S^{t}_{r,p}(\mathbb{R}^{n})}\|g\|_{H^{t+r,p}(\mathbb{R}^{n})}.
\end{eqnarray*}
Hence we can get $X^{t}_{r,p}(\mathbb{R}^{n})\subset
M^{t}_{r,p}(\mathbb{R}^{n})\cap S^{t}_{r,p}(\mathbb{R}^{n})$.

Now we begin to prove the above inequality. At first, we give a
wavelet decomposition of the product of $f h$. Denote
$$\mathbb{Z}^{n}_{N}=\left\{l=(l_{1},\cdots,l_{n})\in\mathbb{Z}^{n},
0\leq l_{i}\leq 2^{N}-1, i=1,\cdots, n\right\}.$$ For
$\varepsilon\in E_{n}$, $l\in\mathbb{Z}^{n}$, $|l|\leq 2^{(M+2)n}$
and $(\varepsilon', l)\in\Lambda_{\varepsilon, n}$, we denote
\begin{eqnarray*}
T_{1,\varepsilon}(x)&=&\sum_{j\in\mathbb{N}, k\in\mathbb{Z}^{n}
 }~\sum_{l\in\mathbb{Z}^{n}}f^{0}_{j,k+l}h^{\varepsilon}_{j,k}\Phi^{0}_{j,k+l}(x)\Phi^{\varepsilon}_{j,k}(x);\\
T_{2,0,\varepsilon, \varepsilon}(x)&=&\sum_{j\in\mathbb{N},
k\in\mathbb{Z}^{n}}~\sum_{l\in\mathbb{Z}^{n}\setminus\{0\}}f^{\varepsilon}_{j,k+l}h^{\varepsilon}_{j,k}\Phi^{\varepsilon}_{j,k+l}(x)\Phi^{\varepsilon}_{j,k}(x);\\
T_{3,\varepsilon}(x)&=&\sum_{j\in\mathbb{N}, k\in\mathbb{Z}^{n}
}~\sum_{l\in\mathbb{Z}^{n}_{N}}~\sum_{l'\in\mathbb{Z}^{n}}
f^{\varepsilon}_{j+N,2^{N}k+l}h^{0}_{j,k+l'}\Phi^{0}_{j,k+l'}(x)\Phi^{\varepsilon}_{j+N,2^{N}k+l}(x);\\
T_{4,\varepsilon}(x)&=&\sum_{j\in\mathbb{N},
k\in\mathbb{Z}^{n}}f^{\varepsilon}_{j,k}h^{\varepsilon
}_{j,k}2^{jn}(\Phi^{\varepsilon}_{j,k}(x))^{2};\\
T_{5}(x)&=&\sum_{k\in
\mathbb{Z}^{n}}~\sum_{l\in\mathbb{Z}^{n}}f^{0}_{0,k}h^{0}_{0,k+l}\Phi^{0}_{0,k}(x)\Phi^{0}_{0,
k+l}(x).
\end{eqnarray*}
Further,  for any $ \varepsilon, \varepsilon'\in E_{n}$, $0<s\leq N$
or $s=0$ and $\varepsilon\neq\varepsilon'$, denote
$$T_{2,\varepsilon, \varepsilon', s}(x)=\sum_{j\in\mathbb{N}, k\in\mathbb{Z}^{n}}\sum_{l\in\mathbb{Z}^{n}}
f^{\varepsilon'}_{j+s,2^{s}k+l}h^{\varepsilon}_{j,k}\Phi^{\varepsilon'}_{j+s,2^{s}k+l}(x)\Phi^{\varepsilon}_{j,k}(x).$$
By the formulas (\ref{for1}) and (\ref{for2}), we can divide
$f(x)h(x)$ into the sum of the above five terms, that is,
\begin{eqnarray*}
f(x)h(x)&=&\sum_{\varepsilon\in E_{n}}T_{1,\varepsilon}+\sum_{0\leq
s\leq N, \varepsilon,\varepsilon'\in E_{n}}T_{2,
s,\varepsilon,\varepsilon'}(x) +\sum_{\varepsilon\in
E_{n}}T_{3,\varepsilon}+\sum_{\varepsilon\in
E_{n}}T_{4,\varepsilon}+T_{5}(x)\\
&:=&\sum^{5}_{i=1}T_{i}(x).
\end{eqnarray*}
 If $g(x)\in H^{t+r,p}(\mathbb{R}^{n})$, $g(x)=\sum\limits_{(\varepsilon,j,k)\in\Lambda_{n}}g^{\varepsilon}_{j,k}\Phi^{\varepsilon}_{j,k}(x)$.
For $\varepsilon,\varepsilon'\in E_{n}$ and $0\leq s\leq N$, we
define $T_{1,\varepsilon}=\int T_{1,\varepsilon}(x)g(x)dx$ and
 $T_{2, s, \varepsilon,\varepsilon'}=\int T_{2, s, \varepsilon,\varepsilon'}(x)g(x)dx$.
 Let
 $$T_{1,1,\varepsilon,
l}=\sum_{j\in\mathbb{N}, k\in\mathbb{Z}^{n}}~\sum_{j\geq j'\geq 0,
\varepsilon',
k'\in\mathbb{Z}^{n}}|f^{0}_{j,k+l}||h^{\varepsilon}_{j,k}||a^{\varepsilon,
\varepsilon'}_{j,k, l, j', k'}||g^{\varepsilon'}_{j',k'}|,$$
$$T_{1,2,\varepsilon,
l}=\sum_{j\in\mathbb{N}, k\in\mathbb{Z}^{n}}~\sum_{j< j',
\varepsilon',
k'\in\mathbb{Z}^{n}}|f^{0}_{j,k+l}||h^{\varepsilon}_{j,k}||a^{\varepsilon,
\varepsilon'}_{j,k, l, j', k'}||g^{\varepsilon'}_{j',k'}|,$$
$$T_{2, 1, s, \varepsilon, \varepsilon',\ l}=\sum_{j\in\mathbb{N}, k\in\mathbb{Z}^{n}}~\sum_{j\geq j'>0,\ \varepsilon'',\ k'\in\mathbb{Z}^{n}}
|f^{\varepsilon'}_{j+s,\
2^{s}k+l}||h^{\varepsilon}_{j,k}||a^{\varepsilon, \varepsilon',
\varepsilon'', s}_{j,k, l, j', k'}||g^{\varepsilon''}_{j',\ k'}|$$
and
$$T_{2,\ 2,\ s,\ \varepsilon, \varepsilon',\ l}=\sum_{j\in\mathbb{N}, k\in\mathbb{Z}^{n}}~\sum_{j<j',\ \varepsilon'',\ k}
|f^{\varepsilon'}_{j+s,\
2^{s}k+l}||h^{\varepsilon}_{j,k}||a^{\varepsilon, \varepsilon',
\varepsilon'', s}_{j,k, l, j',
k'\in\mathbb{Z}^{n}}||g^{\varepsilon''}_{j',\ k'}|.$$ It is easy to see that
 $$ |T_{1,\varepsilon}|\leq \sum\limits_{l\in\mathbb{Z}^{n}}\left(T_{1,1,\varepsilon,
l}+T_{1,2,\varepsilon, l}\right)\text{ and } |T_{2, s,\varepsilon,\varepsilon'}|\leq
\sum\limits_{l\in\mathbb{Z}^{n}}\left(T_{2,1, s, \varepsilon,\varepsilon' l}+T_{2,2, s, \varepsilon,
\varepsilon', l}\right). $$

Let $S_{t}g_{j'}(x)=\sum\limits_{\varepsilon',
k'}2^{j'({n}/{2}+t)}g^{\varepsilon'}_{j',k'}\chi(2^{j'}x-k')$.
For fixed $x$, there is only one $k'$ such that
$\chi(2^{j'}x-k')\neq0$ and the number of $\varepsilon'$ in the sum
is finite. Then the operator $S_{t}g_{j'}(x)$ is equivalent to the
following one:
$$S_{t}g_{j'}(x)=\left(\sum_{\varepsilon',k'}2^{2j'({n}/{2}+t)}|g^{\varepsilon'}_{j',k'}|^{2}\chi(2^{j'}x-k')\right)^{1/2}.$$

Let $M$ be the Hardy-Littlewood maximal function. Then if $j<j'$,
$\forall x\in Q_{j,k}$, we have
$$\sum_{k'}(1+|k-2^{j-j'}k'|)^{-N_{2}}2^{j'(-{n}/{2}+t)}g^{\varepsilon'}_{j',k'}\leq C2^{-jn}MS_{t}g_{j'}(x).$$
Now we estimate the quantities $T_{1,1,\varepsilon,l}$,
$T_{1,2,\varepsilon,l}$, $T_{2,1,s, \varepsilon, \varepsilon',l}$
and $T_{2,2,s, \varepsilon, \varepsilon',l}$ separately.

(1) For $j'\leq j$, we have
$$T_{1,1,\varepsilon, l}=\sum_{j,k}\sum_{j\geq j', \varepsilon', k'}|f^{0}_{j,k+l}||h^{\varepsilon}_{j,k}|
|a^{\varepsilon, \varepsilon'}_{j,k, l,
j',k'}||g^{\varepsilon'}_{j',k'}|.$$
By Lemma \ref{le9}, we know
$$|a^{\varepsilon,\varepsilon'}_{j, k, l, j', k'}|\leq
(1+|l|)^{-N_{1}}(1+|k'-2^{j'-j}k|)^{-N_{2}}2^{{nj'}/{2}+j'-j}.$$ By Lemma \ref{le5}, $f\in
M^{t}_{r,p}(\mathbb{R}^{n})\subset
BMO^{r}(\mathbb{R}^{n})$ implies
$|f^{\varepsilon}_{j,k}|\leq C2^{(r-{n}/{2})j}$. Now we can get
\begin{eqnarray*}
&&\sum_{l\in \mathbb{Z}^{n}}T_{1,1,\varepsilon,l}\leq\sum_{l\in \mathbb{Z}^{n}}\sum_{k, j\geq j'}\int
2^{{3nj'}/{2}+(j'-j)+(r-{n}/{2})j}\sum_{k'}\left(\sum_{\varepsilon'}|g^{\varepsilon'}_{j',k'}|^{2}\right)^{1/2}
\frac{|h^{\varepsilon}_{j,k}|\chi(2^{j'}x-k')}{(1+|l|)^{N_{1}}(1+|k'-2^{j'-j}k|)^{N_{2}}}dx\\
&\leq&\sum_{ j\geq
j'}\int2^{{3nj'}/{2}+(j'-j)+(r-{n}/{2})j}\sum_{k'}\left(\sum_{\varepsilon'}2^{2j'({n}/{2}+t+r)}|g^{\varepsilon'}_{j',k'}|^{2}
\right)^{1/2}2^{-j'(t+r+{n}/{2})}\\
&\times&\sum_{k}\frac{|h^{\varepsilon}_{j,k}|\chi(2^{j'}x-k')}
{(1+|k'-2^{j'-j}k|)^{N_{2}}}2^{j(-{n}/{2}-t)}2^{j({n}/{2}+t)}dx\\
&\leq&\int\sum_{j\geq
j'}2^{{3nj'}/{2}-j'(t+r+{n}/{2})+(j'-j)+(r-{n}/{2})j+j({n}/{2}+t)}2^{-j'n}
S_{t+r}g_{j'}(x)M(S_{-t}h_{j})(x)dx.
\end{eqnarray*}
Because $0<t+r<1$,
\begin{eqnarray*}
\sum_{l\in\mathbb{Z}^{n}}T_{1,1,\varepsilon,l}&\leq&\sum_{j\geq
j'}2^{(j'-j)+(r+t)j-j'(r+t)}\int
S_{t+r}g_{j'}(x)M(S_{-t}h_{j})(x)dx\\
&\leq&\sum_{j\geq
j'}2^{(j'-j)(1-r-t)}\|g\|_{H^{t+r,p}(\mathbb{R}^{n})}\|h\|_{H^{-t,p'}(\mathbb{R}^{n})}\\
&\leq&C\|g\|_{H^{t+r,p}(\mathbb{R}^{n})}\|h\|_{H^{-t,p'}(\mathbb{R}^{n})}.
\end{eqnarray*}
Now we estimate the term
$$T_{2,1,s,\varepsilon,\varepsilon',l}(x)=\sum_{j,k}\sum_{j\geq j'>0}|f^{\varepsilon'}_{j,k}||h^{\varepsilon}_{j+s, 2^{s}k+l}|
|a^{\varepsilon,\varepsilon',\varepsilon'',s}_{j,k,l,j',k'}||g^{\varepsilon''}_{j',k'}|.$$
Because $f\in M^{t}_{r,p}(\mathbb{R}^{n})\subset
BMO^{r}(\mathbb{R}^{n})$, by Lemma \ref{le5},  we have
$|f^{\varepsilon'}_{j,k+l}|\leq C2^{(r-{n}/{2})j}$. By Lemma
\ref{le9},
$$|a^{\varepsilon,\varepsilon',\varepsilon'',s}_{j,k,l,j',k'}|\leq
C(1+|l|)^{-N_{1}}(1+|k'-2^{j'-j}k|)^{-N_{2}}2^{{nj'}/{2}+(j'-j)}.$$ We can get, similarly,
\begin{eqnarray*}
\sum_{l\in\mathbb{Z}^{n}_{N}}T_{2,1,s,\varepsilon,\varepsilon',l}(x)&=&\sum_{l\in\mathbb{Z}^{n}_{N}}\sum_{j,k}\sum_{j\geq
j'>0}|f^{\varepsilon'}_{j,k}||h^{\varepsilon}_{j+s, 2^{s}k+l}|
|a^{\varepsilon,\varepsilon',\varepsilon'',s}_{j,k,l,j',k'}||g^{\varepsilon''}_{j',k'}|\\
&\leq&C\sum_{l\in\mathbb{Z}^{n}_{N}}\sum_{j,k}\sum_{j\geq
j'>0}2^{(r-{n}/{2})j}2^{{nj'}/{2}+j'-j}\frac{|h^{\varepsilon}_{j+s,2^{s}k+l}||g^{\varepsilon''}_{j',k'}|}
{(1+|l|)^{N_{1}}(1+|k'-2^{j'-j}k|)^{N_{2}}}\\
&\leq&C\sum_{l\in\mathbb{Z}^{n}_{N}}\sum_{j\geq
j'>0}\int2^{{3nj'}/{2}+(j'-j)+(r-{n}/{2})j}\sum_{\varepsilon'',k'}
\frac{|g^{\varepsilon'}_{j',k'}||h^{\varepsilon}_{j+s,2^{s}k+l}|\chi(2^{j'}x-k')}
{(1+|l|)^{N_{1}}(1+|k'-2^{j'-j}k|)^{N_{2}}}dx\\
&\leq&\sum_{j\geq j'>0}2^{(j'-j)(1-t-r)}\int
S_{t+r}(g_{j'})(x)MS_{-t}h_{j+s}(x)dx\\
&\leq&C\|S_{t+r}g_{j'}\|_{L^{p}(\mathbb{R}^{n})}\|MS_{-t}h_{j+s}\|_{L^{p'}(\mathbb{R}^{n})}\\
&\leq
&C\|g\|_{H^{t+r,p}(\mathbb{R}^{n})}\|h\|_{H^{-t,p'}(\mathbb{R}^{n})}.
\end{eqnarray*}
 (2) If $j'>j$,  the  estimates of
the terms $T_{1,2,\varepsilon, l}$ and
$T_{2,2,s,\varepsilon,\varepsilon',l}$ are easier than those of $T_{1,1,\varepsilon, l}$ and
$T_{2,1,s,\varepsilon,\varepsilon',l}$.
For example, we estimate the term
\begin{eqnarray*}
T_{1,2,\varepsilon, l}=\sum_{j,k}\sum_{0\leq j<j',
\varepsilon',k'}|f^{0}_{j,k+l}||h^{\varepsilon}_{j,k}||a^{\varepsilon,\varepsilon'}_{j,k,l,j',k'}||g^{\varepsilon'}_{j',k'}|.
\end{eqnarray*}
Because $f\in M^{t}_{r,p}(\mathbb{R}^{n})\subset
BMO^{r}(\mathbb{R}^{n})$, by Lemma \ref{le5},
$|f^{0}_{j,k+l}|\leq C2^{(r-{n}/{2})j}$. By Lemma \ref{le9},
$$|a^{\varepsilon,\varepsilon'}_{j,k,l,j',k'}|\leq
C(1+|l|)^{-N_{1}}(1+|k-2^{j-j'}k'|)^{-N_{2}}2^{-{nj'}/{2}+nj+N(j-j')}.$$
 Because $2^{nj}\int
\chi(2^{j}x-k)dx=1$, we can obtain
\begin{eqnarray*}
&&\sum_{l\in\mathbb{Z}^{n}}T_{1,2,\varepsilon,l}\leq C\sum_{l\in\mathbb{Z}^{n}}\sum_{j,k}\sum_{0\leq
j<j',\varepsilon',k'}2^{(r-{n}/{2})j}2^{-{nj'}/{2}+nj+N(j-j')}\frac{|h^{\varepsilon}_{j,k}||g^{\varepsilon'}_{j',k'}|}
{(1+|l|)^{N_{1}}(1+|k-2^{j-j'}k'|)^{N_{2}}}\\
&\leq&C\sum_{l\in\mathbb{Z}^{n}}\sum_{j,k}\sum_{0\leq
j<j',\varepsilon',k'}\int2^{(r-{n}/{2})j}2^{-{nj'}/{2}+nj+N(j-j')}2^{nj}
\frac{|h^{\varepsilon}_{j,k}||g^{\varepsilon'}_{j',k'}|\chi(2^{j}x-k)}{(1+|l|)^{N_{1}}(1+|k-2^{j-j'}k'|)^{N_{2}}}dx\\
&\leq&C\sum_{\varepsilon,j,k}\int2^{(r-{n}/{2})}2^{-{nj'}/{2}+nj+N(j-j')}2^{nj}2^{-nj}2^{-j'(-{n}/{2}+t+r)}MS_{t+r}(g_{j'})(x)
|h^{\varepsilon}_{j,k}|\chi(2^{j}x-k)dx\\
&\leq&C\sum_{0\leq j<j'}\int 2^{(r-{n}/{2})j}2^{-{nj'}/{2}+nj+N(j-j')}2^{-j({n}/{2}-t)}2^{-j'(-{n}/{2}+t+r)}MS_{t+r}(g_{j'})(x)S_{-t}(h_{j})(x)dx\\
&\leq&C\|g\|_{H^{t+r,p}(\mathbb{R}^{n})}\|h\|_{H^{-t,p'}(\mathbb{R}^{n})}.
\end{eqnarray*}
The estimate for  $T_{2,2,s,\varepsilon,\varepsilon',l}(x)$ can be
obtained similarly. By the same methods used in (1) and (2), we can
get the estimate of the term $T_{5}$. We omit the details.

(3) Now we consider the term  $T_{3,\varepsilon}$. We have the
following claim.

{\it Claim 1:\ Given $r>0,\ t\geq0$ and $t+r<1<p<{n}/{(t+r)}$. If
$f\in BMO^{r}(\mathbb{R}^{n})$, then
$$|\left<T_{3,\varepsilon}(x),g(x)\right>|\leq
C\|g\|_{H^{t+r,p}(\mathbb{R}^{n})}\|h\|_{H^{-t,p'}(\mathbb{R}^{n})}.$$}
In fact, for $l\in \mathbb{Z}^{n}_{N}$, $(\varepsilon,j,k)\in\Lambda_{n}$ and $l'\in\mathbb{Z}^{n}$, let
$$g^{\varepsilon,l'}_{j+N,2^{N}k+l}=2^{-{nj}/{2}}\left<\Phi^{0}_{j,k+l'}(x)\Phi^{\varepsilon}_{j+N,2^{N}k+l}(x), g(x)\right>.$$
We have
\begin{eqnarray*}
|\left<T_{3,\varepsilon}(x),\ g(x)\right>|&=&\left|\sum_{(\varepsilon,j,k)\in\Lambda_{n}}\sum_{l\in \mathbb{Z}_{N}^{n}, l'\in\mathbb{Z}^{n}}2^{{jn}/{2}}h^{0}_{j, k+l'}g^{\varepsilon,l'}_{j+N,2^{N}k+l}f^{\varepsilon}_{j+N, 2^{N}k+l}\right|\\
&\leq&\sum_{l\in\mathbb{Z}^{n}_{N}}\sum_{(\varepsilon,j,k)\in\Lambda_{n}}\left(\sum_{l'\in\mathbb{Z}^{n}}2^{{jn}/{2}}|h^{0}_{j,
k+l'}||g^{\varepsilon,l'}_{j+N,2^{N}k+l}|\right)|f^{\varepsilon}_{j+N,
2^{N}k+l}|.
\end{eqnarray*}
Because  $(F^{r,2}_{1}(\mathbb{R}^{n}))'=BMO^{r}(\mathbb{R}^{n})$,
by Lemma \ref{le8}, we have
\begin{eqnarray*}
&&|\left<T_{3,\varepsilon}(x),\ g(x)\right>|\\
&\leq&C\sum_{l\in\mathbb{Z}^{n}_{N}}\|f\|_{BMO^{r}}\int\left(\sum_{(\varepsilon,j,k)\in\Lambda_{n}}2^{j(n+2r+t)}
\left(\sum_{l'\in\mathbb{Z}^{n}}2^{j({n}/{2}-t)}|h^{0}_{j, k+l'}||g^{\varepsilon,l'}_{j+N,2^{N}k+l}|\right)^{2}\chi(2^{j}x-k)\right)^{1/2}dx\\
&\leq&C\|f\|_{BMO^{r}}\int\left(\sum_{(\varepsilon,j,k)\in\Lambda_{n}}2^{j(n+2r+2t)}
\left(\sum_{l'\in\mathbb{Z}^{n}}(1+|l'|)^{N+n}|g^{\varepsilon,l'}_{j+N,2^{N}k+l}|\right)^{2}\chi(2^{j}x-k)\right)^{1/2}MT^{t}h(x)dx.
\end{eqnarray*}
Because
\begin{eqnarray*}
&&|g^{\varepsilon,l'}_{j+N,2^{N}k+l}|=2^{-{nj}/{2}}\left|\left<\Phi^{0}_{j,k+l'}(x)\Phi^{\varepsilon}_{j+N,2^{N}k+l}(x), \sum_{(\varepsilon',j',k')\in\Lambda_{n}}g^{\varepsilon'}_{j',k'}\Phi(x)^{\varepsilon'}_{j',k'}\right>\right|\\
&=&2^{-{nj}/{2}}\left|\left<\Phi^{0}_{j,k+l'}(x)\Phi^{\varepsilon}_{j+N,2^{N}k+l}(x), \sum_{(\varepsilon',j',k')\in\Lambda_{n}, |j'-j-N|\leq2}g^{\varepsilon'}_{j',k'}\Phi(x)^{\varepsilon'}_{j',k'}\right>\right|,
\end{eqnarray*}
we can get
\begin{eqnarray*}
&&\left(\sum_{(\varepsilon,j,k)\in\Lambda_{n}}2^{(n+2t+2t)j}\left(\sum_{l'\in\mathbb{Z}^{n}}(1+|l'|)^{N+n}|
g^{\varepsilon,l'}_{j+N,2^{N}k+l}|\right)^{2}\chi(2^{j}x-k)\right)^{1/2}
\leq\sum_{|j'-j-N|\leq2}MS_{t+r}g_{j'}(x).
\end{eqnarray*}
Hence we have
\begin{eqnarray*}
&&\int\left(\sum_{(\varepsilon,j,k)\in\Lambda_{n}}2^{j(n+2r+t)}
\left(\sum_{l'\in\mathbb{Z}^{n}}2^{j({n}/{2}-t)}|h^{0}_{j, k+l'}||g^{\varepsilon,l'}_{j+N,2^{N}k+l}|\right)^{2}\chi(2^{j}x-k)\right)^{p/2}dx\\
&\leq&\int\left(\sum_{j\geq0}\left(\sum_{|j'-j-N|\leq2}MS_{t+r}g_{j'}(x)\right)^{2}\right)^{p/2}dx\\
&\leq&C\|g\|_{H^{t+r,p}}.
\end{eqnarray*}
Finally, we obtain, by H\"{o}lder's inequality,
\begin{eqnarray*}
|<T_{3,\varepsilon}(x),\ g(x)>|&\leq&C\|f\|_{BMO^{r}}\|g\|_{H^{t+r,p}}\|MT^{t}h\|_{p'}\\
&\leq&C\|f\|_{BMO^{r}}\|g\|_{H^{t+r,p}}\|h\|_{H^{-t,p'}}.
\end{eqnarray*}
This completes the proof of  Claim 1.

 In order to deal with the term
$T_{4,\varepsilon}(x)
=\sum\limits_{(\varepsilon,j,k)\in\Lambda_{n}}f^{\varepsilon}_{j,k}h^{\varepsilon}_{j,k}\left(\Phi^{\varepsilon}_{j,k}(x)\right)^{2},$
we need the following estimate.

{\it Claim 2:\ For $r>0,\ t\geq 0$ and $t+r<1<p<{n}/{(r+t)}$,
$f(x)\in S^{t}_{r,p}(\mathbb{R}^{n})$ if and only if
$$|\left<T_{4,\varepsilon}(x),\ g(x)\right>|\leq
C\|g\|_{H^{r+t,p}(\mathbb{R}^{n})}\|h\|_{H^{-t,p'}(\mathbb{R}^{n})}.$$}

In fact, for $g\in H^{r+t,p}(\mathbb{R}^{n})$,
$\left<T_{4,\varepsilon}(x),\ g(x)\right>=
\sum\limits_{(\varepsilon,j,k)\in\Lambda_{n}}2^{{nj}/{2}}f^{\varepsilon}_{j,k}g^{\Phi,\varepsilon}_{j,k}h^{\varepsilon}_{j,k}.$
By the definitions of $S^{t}_{r,p}(\mathbb{R}^{n})$ and
$H^{-t,p'}(\mathbb{R}^{n})$, using H\"older's inequality, we have
\begin{eqnarray*}
|\left<T_{4,\varepsilon}(x),\ g(x)\right>|
&\leq&\int\sum_{(\varepsilon,j,k)\in\Lambda_{n},
j\geq0}2^{{3jn}/{2}}|f^{\varepsilon}_{j,k}||g^{\Phi,\varepsilon}_{j,k}||h^{\varepsilon}_{j,k}|\chi(2^{j}x-k)dx\\
&\leq&\int\left(\sum_{(\varepsilon,j,k)\in\Lambda_{n},
j\geq0}2^{j(n+2t)}|g^{\Phi,\varepsilon}_{j,k}|^{2}|f^{\varepsilon}_{j,k}|^{2}\chi(2^{j}x-k)\right)^{1/2}\\
&\times&\left(\sum_{(\varepsilon,j,k)\in\Lambda_{n},
j\geq0}2^{j(n-2t)}|h^{\varepsilon}_{j,k}|^{2}\chi(2^{j}x-k)\right)^{1/2}dx\\
&\leq&\left(\int\left(\sum_{(\varepsilon,j,k)\in\Lambda_{n},
j\geq0}2^{j(n+2t)}|g^{\Phi,\varepsilon}_{j,k}|^{2}|f^{\varepsilon}_{j,k}|^{2}\chi(2^{j}x-k)\right)^{p/2}dx\right)^{1/p}\\
&\times&\left(\int \left(\sum_{(\varepsilon,j,k)\in\Lambda_{n},
j\geq0}2^{j(n-2t)}|h^{\varepsilon}_{j,k}|^{2}\chi(2^{j}x-k)\right)^{p'/2}dx\right)^{1/p'}\\
&\leq&\|f\|_{S^{t}_{r,p}(\mathbb{R}^{n})}\|h\|_{H^{-t,p'}(\mathbb{R}^{n})}.
\end{eqnarray*}
Because $f\in S^{t}_{r,p}(\mathbb{R}^{n})$, we can see that
$$|\left<T_{4,\varepsilon}(x),g(x)\right>|\leq
C\|g\|_{H^{t+r,p}(\mathbb{R}^{n})}\|h\|_{H^{-t,p'}(\mathbb{R}^{n})}.$$

Conversely, let
$\tau^{\varepsilon}_{j,k}=f^{\varepsilon}_{j,k}2^{j({n}/{2}+t)}g^{\Phi,\varepsilon}_{j,k}$
and $\tau(x)=\sum\limits_{(\varepsilon,j,k)\in\Lambda_{n},j\geq
0}~\tau^{\varepsilon}_{j,k}\Phi^{\varepsilon}_{j,k}(x)$.  Denote
$h=|\tau|^{p-2}\overline{\tau}$. For
$h(x)=\sum\limits_{(\varepsilon,j,k)\in\Lambda_{n}}h^{\varepsilon}_{j,k}\Phi^{\varepsilon}_{j,k}(x)$,
we write $h_{t}(x)$  as the function
$$h_{t}(x)=\sum\limits_{(\varepsilon,j,k)\in\Lambda_{n},
j\geq0}2^{jt}h^{\varepsilon}_{j,k}\Phi^{\varepsilon}_{j,k}(x)
:=\sum\limits_{(\varepsilon,j,k)\in\Lambda_{n}}(h_{t})^{\varepsilon}_{j,k}\Phi^{\varepsilon}_{j,k}(x).$$
It is easy to see that $h\in L^{p'}(\mathbb{R}^{n})$ is equivalent to
$h_{t}\in H^{-t,p'}(\mathbb{R}^{n})$.


By the wavelet characterization of $H^{t,p}(\mathbb{R}^{n})$, we get
\begin{eqnarray*}
&&\int\left(\sum_{(\varepsilon,j,k)\in\Lambda_{n},
j\geq0}2^{j(n+2t)}|g^{\Phi,\varepsilon}_{j,k}|^{2}|f^{\varepsilon}_{j,k}|^{2}\chi(2^{j}x-k)\right)^{p/2}dx\\
&=&\int\left(\sum_{(\varepsilon,j,k)\in\Lambda_{n},
j\geq0}2^{jn}|\tau^{\varepsilon}_{j,k}|^{2}\chi(2^{j}x-k)\right)^{p/2}dx\\
&=&\int |\tau|^{p}dx=\int\tau h dx\\
&=&\sum_{(\varepsilon,j,k)\in\Lambda_{n}}2^{j(t+{n}/{2})}f^{\varepsilon}_{j,k}g^{\Phi,\varepsilon}_{j,k}h^{\varepsilon}_{j,k}\\
&=&\sum_{(\varepsilon,j,k)\in\Lambda_{n}}2^{{jn}/{2}}f^{\varepsilon}_{j,k}g^{\Phi,\varepsilon}_{j,k}(h_{t})^{\varepsilon}_{j,k}.
\end{eqnarray*}
Further, we can deduce that
\begin{eqnarray*}
\int\left(\sum_{(\varepsilon,j,k)\in\Lambda_{n},
j\geq0}2^{j(n+2t)}|g^{\Phi,\varepsilon}_{j,k}|^{2}|f^{\varepsilon}_{j,k}|^{2}\chi(2^{j}x-k)\right)^{p/2}dx
&=&\left|\left<T_{4,\varepsilon}(x),\ g(x)\right>\right|\\
&\leq&C\|g\|_{H^{t+r,p}(\mathbb{R}^{n})}\|h_{t}\|_{H^{-t,p'}(\mathbb{R}^{n})}\\
&\leq&C\|g\|_{H^{t+r,p}(\mathbb{R}^{n})}\|\tau\|^{p/p'}_{L^{p}(\mathbb{R}^{n})}.
\end{eqnarray*}
Hence $\|\tau\|_{L^{p}}<\infty$ and $f(x)\in
S^{t}_{r,p}(\mathbb{R}^{n})$. This completes the proof of (i) of
this theorem.

For the proof of (ii), similarly, we divide the product $f(x)h(x)$
into the following terms
\begin{eqnarray*}
T_{1,\varepsilon}(x)&=&\sum_{j\in\mathbb{N}, k\in\mathbb{Z}^{n}}\sum_{l\in\mathbb{Z}^{n}}f^{0}_{j,k+l}h^{\varepsilon}_{j,k}\Phi^{0}_{j,k+l}(x)\Phi^{\varepsilon}_{j,k}(x);\\
T_{2,0,\varepsilon, \varepsilon}(x)&=&\sum_{j\in\mathbb{Z},
k\in\mathbb{Z}^{n}}~\sum_{l\in\mathbb{Z}^{n},l\neq0}f^{\varepsilon}_{j,k+l}h^{\varepsilon}_{j,k}\Phi^{\varepsilon}_{j,k+l}(x)\Phi^{\varepsilon}_{j,k}(x)\\
&+&\sum_{j\in\mathbb{N}, k\in\mathbb{Z}^{n}}f^{\varepsilon}_{j,k}h^{\varepsilon}_{j,k}((\Phi^{\varepsilon})_{j,k}(x))^{2}-2^{jn}\Phi(2^{j}x-k));\\
T_{3,\varepsilon}(x)&=&\sum_{j\in\mathbb{N},
k\in\mathbb{Z}^{n}}~\sum_{l\in\mathbb{Z}^{n}_{N}}\sum_{l'\in\mathbb{Z}^{n}}
f^{\varepsilon}_{j+N,2^{N}k+l}h^{0}_{j,k+l'}\Phi^{0}_{j,k+l'}(x)\Phi^{\varepsilon}_{j+N,2^{N}k+l}(x);\\
T_{4,\varepsilon}(x)&=&\sum_{j\in\mathbb{N},
k\in\mathbb{Z}^{n}}f^{\varepsilon}_{j,k}h^{\varepsilon
}_{j,k}2^{jn}\Phi(2^{j}x-k);\\
T_{5}(x)&=&\sum_{k\in \mathbb{Z}^{n}
}~\sum_{l\in\mathbb{Z}^{n}}f^{0}_{0,k}h^{0}_{0,k+l}\Phi^{0}_{j,k}(x)\Phi^{0}_{0,
k+l}(x).
\end{eqnarray*}
For $ \varepsilon, \varepsilon'\in E_{n}$,  $0<s\leq N$ or $s=0$ and
$\varepsilon\neq\varepsilon'$, denote
$$T_{2, s, \varepsilon, \varepsilon'}(x)=\sum_{j\in\mathbb{N},
k\in\mathbb{Z}^{n}}\sum_{l\in\mathbb{Z}^{n}}f^{\varepsilon'}_{j+s,2^{s}k+l}h^{\varepsilon}_{j,k}\Phi^{\varepsilon'}_{j+s,2^{s}k+l}(x)\Phi^{\varepsilon}_{j,k}(x).$$
By  the same method used in the proof of (i), we can get the conclusion.
\end{proof}

\section{A logarithmic condition for multipliers}
By Lemma \ref{le3}, we know that the multiplier space
$X^{t}_{r,p}(\mathbb{R}^{n})\subset M^{t}_{r,p}(\mathbb{R}^{n})$. In
this section, we consider the reverse inclusion relation. At first
we introduce a logarithmic Morrey spaces.
\begin{definition}\label{def7}
Fix $1<p<{n}/{(r+t)}$ and $\tau\geq0$. We say $f(x)\in
M^{t,\tau}_{r,p}(\mathbb{R}^{n})$ if
$\sup\limits_{|Q|=1}|f_{t,Q}|\leq C$ and
$$\int_{Q}\left|(I-\Delta)^{\frac{t}{2}}f(x)-f_{t,Q}\right|^{p}dx
\leq C\left|1-\log_{2}|Q|\right|^{p\tau }|Q|^{1-{p(r+t)}/{n}},$$
for any cube $Q$ with $|Q|\leq 1$.
\end{definition}

Similar to Theorem \ref{th7},  we have the following wavelet
characterization of $M^{t,\tau}_{r,p}(\mathbb{R}^{n})$.
\begin{theorem}\label{th8}
Given $t,\tau\geq 0, r>0$, $1<p<{n}/{(r+t)}$.
$f(x)=\sum\limits_{(\varepsilon,j,k)\in\Lambda_{n}}f^{\varepsilon}_{j,k}\Phi^{\varepsilon}_{j,k}(x)$
belongs to the logarithmic Morrey spaces
$M^{t,\tau}_{r,p}(\mathbb{R}^{n})$ if and only if
$$\int_{Q}\left(\sum_{\varepsilon\in E_{n}, Q_{j,k}\subset
Q}2^{j(n+2t)}|f^{\varepsilon}_{j,k}|^{2}\chi(2^{j}x-k)\right)^{p/2}dx
\leq C (1-\log_{2}|Q|)^{-p\tau}|Q|^{1-{p(r+t)}/{n}},$$ where
$Q\in\Omega$ with $|Q|\leq 1$.
\end{theorem}
\begin{proof}
Similar to that of Theorem \ref{th7}, the proof of this theorem can
be obtained by the characterization of Triebel-Lizorkin spaces (see
Lemma \ref{le2} ) . We omit the detail.
\end{proof}

In \cite{Fe},  C. Fefferman established the following relation:
$$M^{t}_{r,q}(\mathbb{R}^{n})\subset
X^{t}_{r,p}(\mathbb{R}^{n})\subset M^{t}_{r,p}(\mathbb{R}^{n}),$$
where $q>p>1$.  In this section, we use wavelet characterization to
give a logarithmic type inclusion. Let $r>0$, $t\geq0$,
$1<p<{n}/{(r+t)}$ and $\tau>{1}/{p'}$. We prove that
$M^{t,\tau}_{r,p}(\mathbb{R}^{n})$ is a subspace of
$X^{t}_{r,p}(\mathbb{R}^{n})$ in Theorem \ref{th5}. Hence, for $
q>p$,
$$M^{t}_{r,q}(\mathbb{R}^{n})\subset M^{t,\tau}_{r,p}(\mathbb{R}^{n})\subset X^{t}_{r,p}(\mathbb{R}^{n})
\subset M^{t}_{r,p}(\mathbb{R}^{n})=M^{t,0}_{r,p}(\mathbb{R}^{n}).$$

\begin{lemma}\label{le10}
If $\tau>0$, $r>0$, $t\geq0$ and $1<p<{n}/{(r+t)}$, $f\in
M^{t,\tau}_{r,p}(\mathbb{R}^{n})$ implies that
$|f^{\varepsilon}_{j,k}|\leq C2^{j(r-{n}/{2})}(1+j)^{-\tau}$.
\end{lemma}
\begin{proof}
Because $f\in M^{t,\tau}_{r,p}(\mathbb{R}^{n})$, then for any
dyadic cube $Q\in \Omega$,
\begin{eqnarray*}
&&\int_{Q}\left(\sum_{\varepsilon\in E_{n}, Q_{j,k}\subset
Q}2^{j(n+2t)}|f^{\varepsilon}_{j,k}|^{2}\chi(2^{j}x-k)\right)^{p/2}dx
\leq C(1-\log_{2}|Q|)^{-p\tau}|Q|^{1-{p(r+t)}/{n}}.
\end{eqnarray*}
We have
\begin{eqnarray*}
&&\int_{Q}\left(\sum_{\varepsilon\in E_{n}, Q_{j,k}\subset
Q}2^{j(n-2r)}|f^{\varepsilon}_{j,k}|^{2}\chi(2^{j}x-k)\right)^{p/2}dx\\
&\leq&C|Q|^{{p(r+t)}/{n}}\int_{Q}\left(\sum_{\varepsilon\in
E_{n}, Q_{j,k}\subset
Q}2^{j(n+2t)}|f^{\varepsilon}_{j,k}|^{2}\chi(2^{j}x-k)\right)^{p/2}dx\\
&\leq&C|Q|^{{p(r+t)}/{n}}|Q|^{1-{p(r+t)}/{n}}
(1-\log_{2}|Q|)^{-\tau p}\\
&\leq&C|Q|(1-\log_{2}|Q|)^{-\tau p}.
\end{eqnarray*}
For $\varepsilon\in E_{n}$, $j\in\mathbb{N}$ and $k\in
\mathbb{Z}^{n}$, take $Q=Q_{j,k}$. By the wavelet characterization of
$BMO^{r}(\mathbb{R}^{n})$, we get
\begin{eqnarray*}
\int\left(2^{jn-2jr}|f^{\varepsilon}_{j,k}|^{2}\chi(2^{j}x-k)\right)^{p/2}dx
&\leq&C2^{-jn}(1-\log_{2}2^{-jn})^{-\tau p}
\leq C2^{-jn}(1+j)^{-\tau p}.
\end{eqnarray*}
Then we have
$|Q_{j,k}|(2^{j({n}/{2}-r)}|f^{\varepsilon}_{j,k}|)^{p}\leq
C2^{-jn}(1+j)^{-\tau p}$ and $|f^{\varepsilon}_{j,k}|\leq
C2^{j(r-{n}/{2})}(1+j)^{-\tau}.$

 When $\varepsilon=0$,
\begin{eqnarray*}
f^{0}_{j,k}&=&\left<\sum_{(\varepsilon',j,k)\in\Lambda_{n}}f^{\varepsilon'}_{j',k'}\Phi^{\varepsilon'}_{j',k'},\
\Phi^{0}_{j,k}\right>
=\left<\sum_{j'<j}f^{\varepsilon'}_{j',k'}\Phi^{\varepsilon'}_{j',k'},\
\Phi^{0}_{j,k}\right>.
\end{eqnarray*}
Since
$\left|\sum\limits_{j'<j}f^{\varepsilon'}_{j',k'}\Phi^{\varepsilon'}_{j',k'}(x)\right|\leq
C2^{(r-{n}/{2})j}(1+j)^{-\tau},$ then
$$|f^{0}_{j,k}|\leq C\left<2^{rj}(1+j)^{-\tau},\ |\Phi^{0}_{j,k}|\right>\leq C2^{(r-{n}/{2})j}(1+j)^{-\tau}.$$
\end{proof}
For $\beta=(\beta_{1}, \beta_{2},\cdots,\beta_{n})$, $\beta_{i}\in\mathbb{N}_{+}$, define $f_{\beta}(x)=({\partial}/{\partial x})^{\beta}f$. We
have the following two lemmas.
\begin{lemma}\label{le11}
For $\tau>0$, $r>0$, $t\geq0$ and $1<p<{n}/{(r+t)}$, the function
$f\in M^{t,\tau}_{r,p}(\mathbb{R}^{n})$ implies its derivative
$f_{\beta}\in M^{t-|\beta|,\tau}_{r+|\beta|, p}(\mathbb{R}^{n})$,
where $|\beta|=\sum\limits^{n}_{i=1}\beta_{i}$.
\end{lemma}
\begin{proof}
If $f\in M^{t,\tau}_{r,p}(\mathbb{R}^{n})$ and
$f(x)=\sum\limits_{(\varepsilon, j,
k)\in\Lambda_{n}}f^{\varepsilon}_{j,k}\Phi^{\varepsilon}_{j,k}(x)$,
 by Theorem \ref{th8}, we have
$$\int\left(\sum_{\varepsilon\in E_{n}, Q_{j,k}\subset Q}2^{j(n+2t)}|f^{\varepsilon}_{j,k}|^{2}\chi(2^{j}x-k)\right)^{p/2}dx
\leq C(1-\log_{2}|Q|)^{-p\tau}|Q|^{1-{p(r+t)}/{n}}.$$ Denote by
$f^{\varepsilon, \beta}_{j,k}$ the wavelet coefficients of
$f_{\beta}(x)$. We can get $f^{\beta,
\varepsilon}_{j,k}=2^{j|\beta|}f^{\varepsilon}_{j,k}$ and
\begin{eqnarray*}
&&\int\left(\sum_{\varepsilon\in E_{n}, Q_{j,k}\subset
Q}2^{j(n+2t-2|\beta|)}|f^{\beta,\varepsilon}_{j,k}|^{2}\chi(2^{j}x-k)\right)^{p/2}dx\\
&\leq&\int\left(\sum_{\varepsilon\in E_{n}, Q_{j,k}\subset
Q}2^{j(n+2t)}|f^{\varepsilon}_{j,k}|^{2}\chi(2^{j}x-k)\right)^{p/2}dx\\
&\leq&C(1-\log_{2}|Q|)^{-p\tau}|Q|^{1-\frac{p[(r+|\beta|)+(t-|\beta|)]}{n}}.
\end{eqnarray*}
This implies that $f_{\beta}(x)\in
M^{t-|\beta|,\tau}_{r+|\beta|,p}(\mathbb{R}^{n})$.
\end{proof}

To get the sufficient condition for multiplier spaces, we need to
consider carefully the relationship of different dyadic cubes
relative to combination atoms. Because of this reason, we use
Daubechies wavelets in the rest of this section.

If $g_{u}(x)$ is a $(m, 2^{u}, E_{u})-$combination atom, then we
denote the number of biggest dyadic cubes in $E_{u}$ by $i_{1}$.
Denote by $F_{u,1}$ the set $\left\{i\in\mathbb{N}, i=1,\cdots,
i_{1}\right\}$. If $i\in F_{u,1}$, we denote such cube by $Q_{u, 1,
i}$. The volume of $Q_{u, 1, i}$ is denoted by $2^{-nj_{u,1}}$, that
is, $|Q_{u,1,i}|=2^{-nj_{u,1}}$. Denote $E_{u,1}=E_{u}\setminus
(\bigcup\limits_{i\in F_{u,1}}Q_{u,1,i})$.

We denote the number of biggest dyadic cubes in $E_{u,1}$ by
$i_{2}$. Denote by $F_{u,2}$ the set $\left\{i\in\mathbb{N},
i=1,\cdots, i_{2}\right\}$. If $i\in F_{u,2}$, we denote such cube
by $Q_{u,2,i}$. The  volume of $Q_{u,2,i}$ is denoted by
$2^{-nj_{u,2}}$, that is, $|Q_{u,2,i}|=2^{-nj_{u,2}}$. Denote
$E_{u,2}=E_{u,1}\setminus(\bigcup\limits_{i\in F_{u,2}}Q_{u,2,i})$.

We continue this process until there exists some $s$ such that
$E_{u, s+1}$ is empty.  For $s'\geq s+1$, we denote $i_{s'}=0$ and
 $F_{u, s'}$ and $E_{u, s'}$  are empty sets. Otherwise we
continue until infinitely. Then $E_{u}=\bigcup\limits_{s\geq1, i\in
F_{u,s}}Q_{u,s,i}$ and $g_{u}(x)=\sum\limits_{s\geq1, i\in
F_{u,s}}g_{u,s,i}(x)$, where
$g_{u,s,i}(x)=\sum\limits_{Q_{j,k}\subset
Q_{u,s,i}}g^{\varepsilon}_{j,k}\Phi^{\varepsilon}_{j,k}(x)$.

To compute the norm of $f(x)g_{u}(x)$, we need to find out a special
set of dyadic cubes denoted by
$\left\{Q_{u,s,i,k}\right\}_{(s,i,k)\in H_{u}}$ such that
$\text{supp}g_{u}\subset \bigcup\limits_{(s,i,k)\in
H_{u}}Q_{u,s,i,k}$. $g_{u}(x)$ is nearly $L^{\infty}$ function on
$Q_{u,s,i,k}$ and satisfies the estimate of Lemma \ref{le13}.
We divide such process into the following three steps.

Step 1. For $\forall s\geq1$, if $i\in F_{u,s}$ and
$k\in\mathbb{Z}^{n}$ with $|k|<\sqrt{n}2^{(M+3)n}$, we denote
$(i,k)\in G_{u,s}$.  Denote $F_{u}=\left\{(s,i,k),\ s\geq1, (i,k)\in
G_{u,s}\right\}$. For $\forall (s,i,k)\in F_{u}$, we denote
$Q_{u,s,i,k}=2^{-j_{u,s}}k+Q_{u,s,i}$. For $\forall s\geq1$, denote
$\tilde{E}^{0}_{u,s}=\bigcup\limits_{(i,k)\in G_{u,s}}Q_{u,s,i,k}$.
In the next step, we choose a special subcover to the support of $g_{u}(x)$.

Step 2. We define now $H_{u,s}(s\geq 1)$, $\tilde{E}^{1}_{u,s}(s\geq 1)$ and $G^{1}_{u,s}(s\geq 2)$.

For $s=1$, denote $H_{u,1}=G_{u,1}$ and
$\tilde{E}^{1}_{u,s}=\tilde{E}^{0}_{u,s}$.
For $s=2$, denote $(i,k)\in G^{1}_{u,2}$, if there exists $0\leq
j\leq j_{u,1}$, $l\in\mathbb{Z}^{n}$ such that $Q_{j, l}\subset
\bigcup\limits_{|k-k'|\leq \sqrt{n}2^{(M+2)n}}Q_{u, 2, i, k'}$ and
$<g_{u}, \Phi^{0}_{j,l}(x)>\neq0$. We know that
$\bigcup\limits_{(i,k)\in G^{1}_{u,2}}Q_{u,2,i,k}\subset
\tilde{E}^{1}_{u,1}$. Denote $H_{u,2}=G_{u,2}\setminus G^{1}_{u,2}$
and $\tilde{E}^{1}_{u,2}=\bigcup\limits_{(i,k)\in
H_{u,2}}Q_{u,2,i,k}$.

For $s=3$, denote $(i,k)\in G^{1}_{u,3}$, if there exists $0\leq
j\leq j_{u,2}$, $l\in\mathbb{Z}^{n}$ such that
$Q_{j,l}\subset\bigcup\limits_{|k-k'|\leq\sqrt{n}2^{(M+2)n}}Q_{u, 2,
i, k'}$ and $<g_{u}(x),\ \Phi^{0}_{j_{u,s}, l}(x)>\neq0$. We know
that $\bigcup\limits_{(i,k)\in G^{1}_{u,3}}Q_{u,3,i,k}\subset
\bigcup\limits_{1\leq s\leq2}\tilde{E}^{1}_{u,s}$. Denote
$H_{u,3}=G_{u,3}\setminus G^{1}_{u,3}$ and
$\tilde{E}^{1}_{u,3}=\bigcup\limits_{(i,k)\in H_{u,3}}Q_{u,3,i,k}$.

We continue this process until infinity. For $s\geq2$,
maybe, a party of $G^{1}_{u,s}$, $H_{u,s}$ and $\tilde{E}^{1}_{u,s}$ are empty set.

Step 3.  Let $H_{u}=\left\{(s,i,k),\ s\geq1, (i,k)\in H_{u,s}\right\}$. It is
easy to see that the support of $g_{u}(x)$ is contained in
$\bigcup\limits_{s\geq1}\tilde{E}^{1}_{u,s}$.

For a $(t+r, 2^{u}, E_{u})$-combination atom $g_{u}(x)$ and $g^{0}_{j,k}=<g_{u},
\Phi^{0}_{j,k}>$, we have the following estimate.
\begin{lemma}\label{le13}
Given $r>0$, $t\geq0$, $1<p<{n}/{(t+r)}$ and $s\geq1$. For
$(i,m')\in H_{u,s}$ and $Q_{j,k}\subset Q_{u,s,i,m'}$, we have
$|g^{0}_{j,k}|\leq C2^{u-{nj}/{2}}2^{-(t+r)j_{u,s}}.$
\end{lemma}
\begin{proof}
By the definition of $H_{u,s}$, we have
\begin{eqnarray*}
g^{0}_{j,k}&=&\left<g_{u}(x),\ \Phi^{0}_{j,k}(x)\right>\\
&=&\left<\sum_{(\varepsilon', j', k')\in\Lambda_{n}}g^{\varepsilon'}_{j',k'}\Phi^{\varepsilon'}_{j',k'}(x),\ \Phi^{0}_{j,k}(x)\right>\\
&=&\left<\sum_{(\varepsilon', j', k')\in\Lambda_{n}}\sum_{j'\geq j_{u,s}}g^{\varepsilon'}_{j',k'}\Phi^{\varepsilon'}_{j',k'}(x),\ \Phi^{0}_{j,k}(x)\right>.
\end{eqnarray*}
Because $g_{u}$ is a $(m, 2^{u},
E_{u})$-combination atom, $S_{t+r}(g_{u})(x)\leq C2^{u}$. Hence for every $(\varepsilon', j', k')$, $|g^{\varepsilon'}_{j',k'}|\leq C2^{u}2^{-j'(r+t+{n}/{2})}$. We can obtain
\begin{eqnarray*}
|g^{0}_{j,k}|&=&\left|\left<\sum_{(\varepsilon', j', k')\in\Lambda_{n}}\sum_{j'\geq j_{u,s}}g^{\varepsilon'}_{j',k'}\Phi^{\varepsilon'}_{j',k'}(x),\ \Phi^{0}_{j,k}(x)\right>\right|\\
&\leq&C2^{u}\int\sum_{j'\geq j_{u,s}}2^{-j'(r+t)}2^{-jn'/2}2^{j'n/2}2^{jn/2}|\Phi^{\varepsilon'}(2^{j'}x-k')||\Phi^{0}(2^{j}x-k)|dx\\
&\leq&C2^{u}2^{-jn/2}2^{-j_{u,s}(r+t)}.
\end{eqnarray*}
\end{proof}
\begin{theorem}\label{th11}
Suppose that $\tau>{1}/{p'}$, $t\geq 0,r>0$ and
$1<p<{n}/{(r+t)}$. If $f\in M^{t,\tau}_{r,p}(\mathbb{R}^{n})$ and
$g_{u}$ is a $(t+r,2^{u}, E_{u})-$combination atom, then for
$s\geq1$, $(i,m')\in H_{u,s}$ and $Q=Q_{u,s,i,m'}$, we have
$\|fg_{u}\|_{H^{t,p}(Q)}\leq C(1+j_{u,s})^{-\tau}2^{u}|Q|^{1/p}.$
\end{theorem}
\begin{proof}

First, for $t\geq 0$, we prove $\|fg_{u}\|_{L^{p}(Q)}\leq
Cj_{u,s}^{-\tau}2^{u}|Q|^{1/p}.$ Let
$f(x)=\sum\limits_{(\varepsilon,j,k)\in\Lambda_{n}}f^{\varepsilon}_{j,k}\Phi^{\varepsilon}_{j,k}(x)$
and
$g_{u}(x)=\sum\limits_{(\varepsilon,j,k)\in\Lambda_{n}}g^{\varepsilon}_{j,k}\Phi^{\varepsilon}_{j,k}(x)$.
Denote by $\Lambda_{n}'$ the set
$$\left\{(\varepsilon,\varepsilon', j,k,l),\ \varepsilon, \varepsilon'\in E_{n},
(\varepsilon,k)\neq(\varepsilon',k+l), j\in\mathbb{Z},
k,l\in\mathbb{Z}^{n}, |l|\leq\sqrt{n}2^{(M+2)n}\right\}.$$ For any
dyadic cube $Q\subset E_{u}$, by the formulas (\ref{for1}) and
(\ref{for2}), we decompose the product $f(x)g_{u}(x)$ into the
following parts.
\begin{eqnarray*}
T_{1}(x)&=&\sum_{(\varepsilon,j,k)\in\Lambda_{n},
j=j_{u,s}}~\sum_{|l|\leq\sqrt{n}2^{(M+2)n}}f^{0}_{j,k+l}g^{0}_{j,k}\Phi^{0}_{j,k+l}(x)\Phi^{0}_{j,k}(x);\\
T_{2}(x)&=&\sum_{(\varepsilon,j,k)\in\Lambda_{n},
j\geq j_{u,s}}~\sum_{|l|\leq\sqrt{n}2^{(M+2)n}}f^{0}_{j,k+l}g^{\varepsilon}_{j,k}\Phi^{0}_{j,k+l}(x)\Phi^{\varepsilon}_{j,k}(x);\\
T_{3}(x)&=&\sum_{(\varepsilon,j,k)\in\Lambda_{n},
j\geq j_{u,s}}~\sum_{|l|\leq\sqrt{n}2^{(M+2)n}}f^{\varepsilon}_{j,k+l}g^{0}_{j,k}\Phi^{\varepsilon}_{j,k+l}(x)\Phi^{0}_{j,k}(x);\\
T_{4}(x)&=&\sum_{(\varepsilon,\varepsilon',j,k, l)\in\Lambda'_{n},
j\geq j_{u,s}}~\sum_{|l|\leq\sqrt{n}2^{(M+2)n}}f^{\varepsilon'}_{j,k+l}g^{\varepsilon}_{j,k}\Phi^{\varepsilon'}_{j,k+l}(x)\Phi^{\varepsilon}_{j,k}(x);\\
T_{5}(x)&=&\sum_{(\varepsilon,j,k)\in\Lambda_{n}, j\geq
j_{u,s}}f^{\varepsilon}_{j,k}g^{\varepsilon}_{j,k}\left((\Phi^{\varepsilon}_{j,k}(x))^{2}-2^{{nj}/{2}}\Phi^{0}_{j,k}(x)\right);\\
T_{6}(x)&=&\sum_{(\varepsilon,j,k)\in\Lambda_{n}, j\geq
j_{u,s}}f^{\varepsilon}_{j,k}g^{\varepsilon}_{j,k}2^{{nj}/{2}}\Phi^{0}_{j,k}(x).
\end{eqnarray*}
We divide the rest of the proof into three steps.

 {\bf Step 1.} For
$t\geq 0$, we estimate the norm $\|\cdot\|_{L^{p}(Q)}$ for the terms
$T_{i}(x), i=1,2,\cdots, 6.$ By the wavelet characterization of
Sobolev spaces, we obtain
$$f=\sum_{(\varepsilon,j,k)\in\Lambda_{n}}f^{\varepsilon}_{j,k}\Phi^{\varepsilon}_{j,k}(x)\in L^{p}\text{ if and only if }
\left\|\left(\sum_{(\varepsilon,j,k)\in\Lambda_{n}}2^{jn}|f^{\varepsilon}_{j,k}|^{2}\chi(2^{j}x-k)\right)^{1/2}\right\|_{L^{p}}.$$
Denote by $S_{0}(f)(x)$  the operator
$\left(\sum\limits_{(\varepsilon,j,k)\in\Lambda_{n}}2^{jn}|f^{\varepsilon}_{j,k}|^{2}\chi(2^{j}x-k)\right)^{1/2}$.
We have $\|f\|_{L^{p}}\simeq \|S_{0}(f)\|_{L^{p}}.$

(1) Because $f\in M^{t,\tau}_{r,p}$, by Lemma \ref{le10},
$|f_{j_{u,s},k}^{0}|\leq(1+j_{u,s})^{-\tau}2^{(r-{n}/{2})j_{u,s}}$.
By Lemma \ref{le13}, we have $|g^{0}_{j_{u,s},k}|\leq
C2^{u-(r+t+{n}/{2})j_{u,s}}$. Hence we can get
\begin{eqnarray*}
&&S_{0}(T_{1})(x)\\
&=&\left(\sum_{|Q_{j,k}|=|Q_{j_{u,s},k}|}2^{2j_{u,s}n}|f^{0}_{j_{u,s},k}|^{2}|g^{0}_{j_{u,s},k}|^{2}\chi(2^{j_{u,s}}x-k)\right)^{1/2}\\
&\leq&\left(\sum_{|Q_{j,k}|=|Q_{j_{u,s},k}|}2^{2j_{u,s}n}(1+j_{u,s})^{-2\tau}
2^{2(r-{n}/{2})j_{u,s}}2^{2u-2(r+t+{n}/{2})j_{u,s}}2^{2j_{u,s}({n}/{2}+t)}\chi(2^{j_{u,s}}x-k)\right)^{1/2}.
\end{eqnarray*}
Because $j_{u,s}\geq0$, we have
$$S_{0}(T_{1})(x)\leq(1+j_{u,s})^{-\tau
}2^{u}\left(\sum_{|Q_{j,k}|=|Q_{j_{u,s},k}|}\chi(2^{j}x-k)\right)^{1/2}.$$
and
$$\|T_{1}\|_{L^{p}(Q)}=\|S_{0}(T_{1})\|_{L^{p}(Q)}\leq C(1+j_{u,s})^{-\tau}2^{u}|Q|^{1/p}.$$

(2) Now we estimate
$T_{2}(x)=\sum\limits_{(\varepsilon,j,k)\in\Lambda_{n}, j\geq
j_{u,s}}~\sum\limits_{|l|\leq\sqrt{n}2^{(M+2)n}}f^{0}_{j,k+l}g^{\varepsilon}_{j,k}\Phi^{0}_{j,k+l}(x)\Phi^{\varepsilon}_{j,k}(x)$.
Because $j\geq j_{u,s}\geq0$, we have
$$|f_{j,k}^{0}|\leq(1+j_{u,s})^{-\tau}2^{(r-{n}/{2})j}\leq (1+j_{u,s})^{-\tau}2^{(r+t-{n}/{2})j}.$$
Let
$$\Lambda_{n,2}^{Q}=\left\{(\varepsilon,j,k)\in\Lambda_{n}\mid j\geq j_{u,s}, \forall l\leq\sqrt{n}2^{(M+2)n},
|{\rm supp}(\Phi^{0}_{j,k+l}(x)\Phi^{\varepsilon}_{j,k}(x))\cap Q|\neq0\right\}.$$
Then we can get
\begin{eqnarray*}
S_{0}(T_{2})(x)&=&\left(\sum_{(\varepsilon,j,k)\in\Lambda_{n,2}^{Q}}|f^{0}_{j,k+l}|^{2}|g^{\varepsilon}_{j,k}|^{2}\chi(2^{j}x-k)\right)^{1/2}\\
&\leq&C\left(\sum_{(\varepsilon,j,k)\in\Lambda_{n,2}^{Q}}2^{2jn}(1+j)^{-2\tau}2^{2j(r+t-{n}/{2})}|g^{\varepsilon}_{j,k}|^{2}\chi(2^{j}x-k)\right)^{1/2}\\
&\leq&C(1+j_{u,s})^{-\tau}S_{t+r}(g_{u})(x).
\end{eqnarray*}
Because $g_{u}(x)$ is a $(t+r, 2^{u}, E_{u})-$combination atom,
\begin{eqnarray*}
\|S_{0}(T_{2})\|_{L^{p}(Q)}&\leq&C(1+j_{u,s})^{-\tau}\|S_{t+r}(g_{u})\|_{L^{p}(Q)}
\leq C(1+j_{u,s})^{-\tau}2^{u}|Q|^{1/p}.
\end{eqnarray*}

(3) Since $g_{u}(x)$ is a $(t+r, 2^{u}, E_{u})-$combination atom, for
$s\geq1$, $(i, m')\in H_{u,s}$, $j\geq0$ and $Q=Q_{u,s,i,m'}$,
$$|g^{0}_{j,k}|\leq C2^{u-{nj}/{2}}2^{-(r+t)j}=2^{u-{nj}/{2}}|Q|^{{(r+t)}/{n}}.$$
Let
$$\Lambda_{n,3}^{Q}=\left\{(\varepsilon,j,k)\in\Lambda_{n}\mid j\geq j_{u,s}, \forall l\leq\sqrt{n}2^{(M+2)n},
|{\rm supp}(\Phi^{0}_{j,k}(x)\Phi^{\varepsilon}_{j,k+l}(x))\cap Q|\neq0\right\}.$$
We have, by $j_{u,s}\geq0$,
\begin{eqnarray*}
S_{0}(T_{3})(x)&=&\left(\sum_{(\varepsilon,j,k)\in\Lambda_{n,3}^{Q}}
2^{jn}|g^{0}_{j,k}|^{2}|f^{\varepsilon}_{j,k}|^{2}\chi(2^{j}x-k)\right)^{1/2}\\
&\leq&C2^{u}|Q|^{{(r+t)}/{n}}\left(\sum_{(\varepsilon,j,k)\in\Lambda_{n,3}^{Q}}
2^{j(n+2t)}|f^{\varepsilon}_{j,k}|^{2}\chi(2^{j}x-k)\right)^{1/2}.
\end{eqnarray*}
By the fact that $j\geq j_{u,s}$ and $f\in M^{t,\tau}_{r,p}(\mathbb{R}^{n})$, we get
\begin{eqnarray*}
\|S_{0}(T_{3})\|_{L^{p}(Q)}\leq
C2^{u}|Q|^{{(r+t)}/{n}}\left(-\log_{2}|Q|\right)^{-\tau}|Q|^{{1}/{p}-{(r+t)}/{n}}\leq
C2^{u}j_{u,s}^{-\tau}|Q|^{1/p}.
\end{eqnarray*}

(4) Now we estimate the term $T_{4}(x)$. Let
$$\Lambda_{n,4}^{Q}=\left\{(\varepsilon,j,k)\in\Lambda_{n}\mid j\geq j_{u,s}, \forall (\varepsilon,\varepsilon',j,k,l)\in\Lambda'_{n},
|{\rm
supp}(\Phi^{\varepsilon'}_{j,k+l}(x)\Phi^{\varepsilon}_{j,k}(x))\cap
Q|\neq0\right\}.$$ Because $f\in M^{t,\tau}_{r,p}(\mathbb{R}^{n})$,
we have
$$|f^{\varepsilon}_{j,k}|\leq C(1+j)^{-\tau}2^{j(r-{n}/{2})}\leq C(1+j)^{-\tau}2^{j(r+t-{n}/{2})}$$ and
\begin{eqnarray*}
S_{0}(T_{4})(x)&=&\left(\sum_{(\varepsilon,j,k)\in \Lambda_{n,4}^{Q}}2^{jn}(1+j)^{-2\tau}2^{j(2r+2t-n)}2^{jn}|g^{\varepsilon}_{j,k}|^{2}\chi(2^{j}x-k)\right)^{1/2}\\
&\leq&C(1+j_{u,s})^{-\tau}S_{t+r}(g_{u})(x).
\end{eqnarray*}
Then we can get, by the fact that and $g_{u}$ is a $(t+r, 2^{u},
E_{u})-$combination atom,
$$\|S_{0}(T_{4})\|_{L^{p}(Q)}\leq C(1+j_{u,s})^{-\tau}\|S_{t+r}(g_{u})\|_{L^{p}(Q)}\leq C2^{u}(1+j_{u,s})^{-\tau}|Q|^{1/p}.$$

(5) Now we estimate  the term
$T_{5}(x)=\sum\limits_{(\varepsilon,j,k)\in\Lambda_{n},j\geq
j_{u,s}}f^{\varepsilon}_{j,k}g^{\varepsilon}_{j,k}\left((\Phi^{\varepsilon}_{j,k}(x))^{2}-2^{{nj}/{2}}\Phi^{0}_{j,k}(x)\right).$
Because the function
$(\Phi^{\varepsilon}_{j,k}(x))^{2}-2^{{nj}/{2}}\Phi^{0}_{j,k}(x)$
plays the role as that of $\Phi^{\varepsilon}_{j,k}(x)$, we have
$\|S_{0}(T_{5})\|_{L^{p}(Q)}\leq
C2^{u}(1+j_{u,s})^{-\tau}|Q|^{1/p}$.

(6) To estimate the term $T_{6}(x)$, we take $h\in L^{
p'}(Q)$.
Let
$$\Lambda_{n,6}^{Q}=\left\{(\varepsilon,j,k)\in\Lambda_{n}\mid j\geq j_{u,s},
|{\rm supp}(\Phi^{0}_{j,k+l}(x)\Phi^{\varepsilon}_{j,k}(x))\cap Q|\neq0\right\}.$$
By the orthogonality of the wavelet function, we have
\begin{eqnarray*}
<T_{6}, h>&=&\left<\sum_{(\varepsilon,
j,k)\Lambda^{Q}_{n,6}}f^{\varepsilon}_{j,k}g^{\varepsilon}_{j,k}2^{{jn}/{2}}\Phi^{0}_{j,k},
h\right>
=\sum_{(\varepsilon,j,k)\in\Lambda_{n,6}^{Q}}f^{\varepsilon}_{j,k}g^{\varepsilon}_{j,k}2^{{nj}/{2}}h^{0}_{j,k}.
\end{eqnarray*}
Then
\begin{eqnarray*}
|<T_{6}, h>|&\leq&\int\sum_{(\varepsilon,j,k)\in\Lambda_{n,6}^{Q}}2^{nj}|f^{\varepsilon}_{j,k}||g^{\varepsilon}_{j,k}|2^{nj/2}
|h^{0}_{j,k}|\chi(2^{j}x-k)dx\\
&\leq&\int\sum_{(\varepsilon,j,k)\in\Lambda_{n,6}^{Q}}2^{nj}|f^{\varepsilon}_{j,k}||g^{\varepsilon}_{j,k}|\chi(2^{j}x-k)M(h)(x)dx.
\end{eqnarray*}
By H\"older's inequality and $j\geq j_{u,s}\geq0$, it can be deduced
that
\begin{eqnarray*}
&&\sum_{(\varepsilon,j,k)\in\Lambda_{n,6}^{Q}}2^{nj}|f^{\varepsilon}_{j,k}||g^{\varepsilon}_{j,k}|\chi(2^{j}x-k)\\
&\leq&\left(\sum_{(\varepsilon,j,k)\in\Lambda_{n,6}^{Q}}2^{(n-2t-2r)j}|f^{\varepsilon}_{j,k}|^{2}\chi(2^{j}x-k)\right)^{1/2}
\left(\sum_{(\varepsilon,j,k)\in\Lambda_{n,6}^{Q}}2^{(n+2t+2r)j}|g^{\varepsilon}_{j,k}|^{2}\chi(2^{j}x-k)\right)^{1/2}\\
&\leq&C2^{u}|Q|^{{(r+t)}/{n}}
\left(\sum_{(\varepsilon,j,k)\in\Lambda_{n,6}^{Q}}2^{(n-2t-2r)j}|f^{\varepsilon}_{j,k}|^{2}\chi(2^{j}x-k)\right)^{1/2}.
\end{eqnarray*}
Finally we can get
\begin{eqnarray*}
|<T_{6}, h>|&\leq&C2^{u}|Q|^{{(r+t)}/{n}}
\left\|\left(\sum_{(\varepsilon,j,k)\in\Lambda_{n,6}^{Q}}
2^{(n-2t-2r)j}|f^{\varepsilon}_{j,k}|^{2}\chi(2^{j}x-k)\right)^{1/2}\right\|_{L^{p}(Q)}\|h\|_{L^{p'}(Q)}\\
&\leq&C2^{u}|Q|^{{(r+t)}/{n}}|1-\log_{2}|Q||^{-\tau}|Q|^{{1}/{p}-{(r+t)}/{n}}\\
&\leq&C2^{u}|Q|^{1/p}(1+j_{u,s})^{-\tau}.
\end{eqnarray*}
Because $h$ takes over all functions in $L^{p'}(Q)$, it is obvious
that $\|S_{0}(T_{6})\|_{L^{p}(Q)}\leq
C(1+j_{u,s})^{-\tau}2^{u}|Q|^{1/p}$.

{\bf Step 2.} Assume that $0\leq t<1$. We need to prove for
$g(x)=\sum_{(\varepsilon,j,k)\in\Lambda_{n}}g^{\varepsilon}_{j,k}\Phi^{\varepsilon}_{j,k}(x)$
and
$S_{t}(g)(x)=\left(\sum\limits_{(\varepsilon,j,k)\in\Lambda_{n}}2^{j(2t+n)}|g^{\varepsilon}_{j,k}|^{2}\chi(2^{j}x-k)\right)^{1/2},$
 $\|S_{t}(fg_{u})\|_{L^{p}(Q)}\leq
C2^{u}(1+j_{u,s})^{-\tau}|Q|^{1/p}$. The index sets $\Lambda_{n,i}^{Q}, i=1,2,\cdots, 6$ are the same as  Step I.

(1) For the term $T_{1}$, we have
\begin{eqnarray*}
&&S_{t}(T_{1})(x)
=\left(\sum_{|Q_{j,k}|=|Q_{j_{u,s},k}|}
2^{j_{u,s}(n+2t)}2^{j_{u,s}n}|f^{0}_{j_{u,s},k}|^{2}|g^{0}_{j_{u,s},k}|^{2}\chi(2^{j_{u,s}}x-k)\right)^{1/2}\\
&\leq&\left(\sum_{|Q_{j,k}|=|Q_{j_{u,s},k}|}
2^{j_{u,s}(n+2t)}2^{j_{u,s}n}(1+j_{u,s})^{-2\tau}
2^{2(r-{n}/{2})j_{u,s}}2^{2u-2(r+t+{n}/{2})j_{u,s}}2^{2j_{u,s}({n}/{2}+t)}\chi(2^{j_{u,s}}x-k)\right)^{1/2}\\
&\leq&(1+j_{u,s})^{-\tau
}2^{u}\left(\sum_{|Q_{j,k}|=|Q_{j_{u,s},k}|}\chi(2^{j}x-k)\right)^{1/2}.
\end{eqnarray*}
Then we have
$\|T_{1}\|_{H^{t,p}(Q)}=\|S_{t}(T_{1})\|_{L^{p}(Q)}\leq C(1+j_{u,s})^{-\tau}2^{u}|Q|^{1/p}.$

(2) For the term
 $$T_{2}(x)=\sum\limits_{(\varepsilon,j,k)\in\Lambda_{n}, j\geq
j_{u,s}}~\sum\limits_{|l|\leq\sqrt{n}2^{(M+2)n}}f^{0}_{j,k+l}g^{\varepsilon}_{j,k}\Phi^{0}_{j,k+l}(x)\Phi^{\varepsilon}_{j,k}(x),$$
we have
\begin{eqnarray*}
S_{t}(T_{2})(x)&=&\left(\sum_{(\varepsilon,j,k)\in\Lambda_{n,2}^{Q}}2^{j(n+2t)}2^{jn}|f^{0}_{j,k+l}|^{2}|g^{\varepsilon}_{j,k}|^{2}\chi(2^{j}x-k)\right)^{1/2}\\
&\leq&C\left(\sum_{(\varepsilon,j,k)\in\Lambda_{n,2}^{Q}}2^{j(n+2t)}2^{jn}(1+j)^{-2\tau}2^{2j(r-{n}/{2})}|g^{\varepsilon}_{j,k}|^{2}\chi(2^{j}x-k)\right)^{1/2}\\
&\leq&C(1+j_{u,s})^{-\tau}S_{t+r}(g_{u})(x).
\end{eqnarray*}
Because $g_{u}(x)$ is a $(t+r, 2^{u}, E_{u})-$combination atom, we
have
\begin{eqnarray*}
&&\|S_{t}(T_{2})\|_{L^{p}(Q)}\leq
C(1+j_{u,s})^{-\tau}\|S_{t+r}(g_{u})\|_{L^{p}(Q)} \leq
C(1+j_{u,s})^{-\tau}2^{u}|Q|^{1/p}.
\end{eqnarray*}

(3) Because $g_{u}(x)$ is a $(t+r, 2^{u}, E_{u})-$combination atom,
for $s\geq1$, $(i, m')\in H_{u,s}$ and $Q=Q_{u,s,i,m'}$, we can get
$|g^{0}_{j,k}|\leq
C2^{u-{nj}/{2}}2^{-(r+t)j}=2^{u-{nj}/{2}}|Q|^{{(r+t)}/{n}}$
and
\begin{eqnarray*}
S_{t}(T_{3})(x)&=&\left(\sum_{(\varepsilon,j,k)\in\Lambda_{n,3}^{Q}}
2^{j(n+2t)}|g^{0}_{j,k}|^{2}|f^{\varepsilon}_{j,k}|^{2}\chi(2^{j}x-k)\right)^{1/2}\\
&\leq&C2^{u}\left(\sum_{(\varepsilon,j,k)\in\Lambda_{n,3}^{Q}}
2^{j(n+2t)}2^{-nj}2^{-(r+t)j}|f^{\varepsilon}_{j,k}|^{2}\chi(2^{j}x-k)\right)^{1/2}\\
&\leq&C2^{u}|Q|^{{(r+t)}/{n}}\left(\sum_{(\varepsilon,j,k)\in\Lambda_{n,3}^{Q}}
2^{j(n+2t)}|f^{\varepsilon}_{j,k}|^{2}\chi(2^{j}x-k)\right)^{1/2}.
\end{eqnarray*}
From the fact that $j\geq j_{u,s}$ and $f\in M^{t,\tau}_{r,p}(\mathbb{R}^{n})$, we
deduce that
\begin{eqnarray*}
\|S_{t}(T_{3})\|_{L^{p}(Q)}\leq
C2^{u}|Q|^{{(r+t)}/{n}}\left(-\log_{2}|Q|\right)^{-\tau}|Q|^{{1}/{p}-{(r+t)}/{n}}\leq
C2^{u}j_{u,s}^{-\tau}|Q|^{1/p}.
\end{eqnarray*}

(4) For the term $T_{4}(x)$, because $f\in M^{t,\tau}_{r,p}(\mathbb{R}^{n})$ and
$g_{u}$ is a $(t+r, 2^{u}, E_{u})-$combination atom, we have
$|f^{\varepsilon}_{j,k}|\leq C(1+j)^{-\tau}2^{j(r-{n}/{2})}$ and
\begin{eqnarray*}
S_{t}(T_{4})(x)&=&\left(\sum_{(\varepsilon,j,k)\in\Lambda_{n,4}^{Q}}
2^{j(n+2t)}(1+j)^{-2\tau}2^{j(2r-n)}2^{jn}|g^{\varepsilon}_{j,k}|^{2}\chi(2^{j}x-k)\right)^{1/2}
\leq C(1+j_{u,s})^{-\tau}S_{t+r}(g_{u})(x).
\end{eqnarray*}
Then we can get
$$\|S_{t}(T_{4})\|_{L^{p}(Q)}\leq C(1+j_{u,s})^{-\tau}\|S_{t+r}(g_{u})\|_{L^{p}(Q)}\leq C2^{u}(1+j_{u,s})^{-\tau}|Q|^{1/p}.$$

(5) Now we estimate  the term
$$T_{5}(x)=\sum_{(\varepsilon,j,k)\in\Lambda_{n},j\geq
j_{u,s}}f^{\varepsilon}_{j,k}g^{\varepsilon}_{j,k}\left((\Phi^{\varepsilon}_{j,k}(x))^{2}-2^{{nj}/{2}}\Phi^{0}_{j,k}(x)\right).$$
Because the function
$(\Phi^{\varepsilon}_{j,k}(x))^{2}-2^{{nj}/{2}}\Phi^{0}_{j,k}(x)$
plays the role as that of $\Phi^{\varepsilon}_{j,k}(x)$, we have
$\|S_{t}(T_{5})\|_{L^{p}(Q)}\leq
C2^{u}(1+j_{u,s})^{-\tau}|Q|^{1/p}$.

(6) For the term $T_{6}(x)$, we take $h\in H^{-t,
p'}(Q)$. By the orthogonality of the wavelet
functions, we have
\begin{eqnarray*}
\left<T_{6}, h\right>&=&\left<\sum_{(\varepsilon,
j,k)\in\Lambda^{Q}_{n,6}}f^{\varepsilon}_{j,k}g^{\varepsilon}_{j,k}2^{{jn}/{2}}\Phi^{0}_{j,k},
h\right>
=\sum_{(\varepsilon,j,k)\in\Lambda_{n,6}^{Q}}f^{\varepsilon}_{j,k}g^{\varepsilon}_{j,k}2^{{nj}/{2}}h^{0}_{j,k}.
\end{eqnarray*}
We can get
\begin{eqnarray*}
\left|\left<T_{6},\
h\right>\right|&\leq&\left|\int\sum_{(\varepsilon,
j,k)\in\Lambda^{Q}_{n}}2^{{nj}/{2}}g^{\varepsilon}_{j,k}f^{\varepsilon}_{j,k}2^{jn}h^{0}_{j,k}\chi(2^{j}x-k)dx\right|\\
&\leq&\int\sum_{(\varepsilon,
j,k)\in\Lambda^{Q}_{n}}2^{nj}|g^{\varepsilon}_{j,k}||f^{\varepsilon}_{j,k}|M(h)(x)\chi(2^{j}x-k)dx.
\end{eqnarray*}
By H\"older's inequality, we have
\begin{eqnarray*}
&&\sum_{(\varepsilon,j,k)\in\Lambda_{n,6}^{Q}}2^{nj}|f^{\varepsilon}_{j,k}||g^{\varepsilon}_{j,k}|\chi(2^{j}x-k)\\
&\leq&\left(\sum_{(\varepsilon,j,k)\in\Lambda_{n,6}^{Q}}2^{(n-2t-2r)j}|f^{\varepsilon}_{j,k}|^{2}\chi(2^{j}x-k)\right)^{1/2}
\left(\sum_{(\varepsilon,j,k)\in\Lambda_{n,6}^{Q}}2^{(n+2t+2r)j}|g^{\varepsilon}_{j,k}|^{2}\chi(2^{j}x-k)\right)^{1/2}\\
&\leq&S_{t+r}(g_{u})(x)
\left(\sum_{(\varepsilon,j,k)\in\Lambda_{n,6}^{Q}}2^{(n-2t-2r)j}|f^{\varepsilon}_{j,k}|^{2}\chi(2^{j}x-k)\right)^{1/2}.
\end{eqnarray*}
Because $j\geq j_{u,s}$ implies that $2^{-j(r+t)}\leq 2^{-j_{u,s}(r+t)}$, we can get
\begin{eqnarray*}
|<T_{6}, h>|&\leq&C2^{u}
\left\|\left(\sum_{(\varepsilon,j,k)\in\Lambda_{n,6}^{Q}}
2^{(n-2t-2r)j}|f^{\varepsilon}_{j,k}|^{2}\chi(2^{j}x-k)\right)^{1/2}\right\|_{L^{p}(Q)}\|h\|_{L^{p'}(Q)}\\
&\leq&C2^{u}|Q|^{{(r+t)}/{n}}|1-\log_{2}|Q||^{-\tau}|Q|^{{1}/{p}-{(r+t)}/{n}}\\
&\leq&C2^{u}|Q|^{1/p}(1+j_{u,s})^{-\tau}.
\end{eqnarray*}
Because $h$ takes over all functions in $H^{-t,p'}(Q)$, we can get
$\|S_{t}(T_{6})\|_{L^{p}(Q)}\leq
C(1+j_{u,s})^{-\tau}2^{u}|Q|^{1/p}$. This completes the proof for
the case $0\leq t<1$.

{\bf Step 3.} Now we consider the case $t\geq1$.
In this case, there exists an integer $[t]$ such that $[t]\leq
t<[t]+1$. For any $\alpha=(\alpha_{1}, \alpha_{2}, \cdots,
\alpha_{n})\in \mathbb{N}^{n}$ with $|\alpha|=\sum\limits_{i=1}^{n}\alpha_{i}=[t]$,
the derivative $\frac{\partial^{\alpha}}{\partial x^{\alpha}}$ of
the product $fg_{u}$ can be represented as
$$\frac{\partial^{\alpha}}{\partial x^{\alpha}}(fg_{u})
=\sum_{\gamma,\beta}C_{\alpha,\beta,\gamma}(\frac{\partial^{\beta}}{\partial
x^{\beta}})f(x)\frac{\partial^{\gamma}}{\partial
x^{\gamma}}g_{u}(x),$$ where $|\alpha|=|\beta|+|\gamma|$.
 Denote $\frac{\partial^{\beta}}{\partial
x^{\beta}}f(x)$ by $f_{\beta}(x)$ and denote
$\frac{\partial^{\gamma}}{\partial x^{\gamma}}g_{u}(x)$ by
$g_{u,\gamma}(x)$. Applying the conclusion in Step 1, we only need
to prove
$$\|f_{\beta}g_{u,\gamma}\|_{H^{t-[t],p}}\leq C\tau_{Q}^{-\tau}2^{u}|Q|^{1/p},$$
that is,
$\|S_{t-[t]}(f_{\beta}g_{u,\gamma})\|_{L^{p}}\leq C2^{u}\tau_{Q}^{-\tau}|Q|^{1/p}.$

If $g_{u}$ is a $(t+r,2^{u}, E_{u})-$combination atom, then
\begin{eqnarray*}
g_{u,\gamma}(x)&=&\sum_{Q_{j,k}\in\mathcal{F}_{v,l}}g^{\varepsilon}_{j,k}(\frac{\partial}{\partial
x})^{\gamma}[2^{nj/2}\Phi^{\varepsilon}_{j,k}(x)]
=\sum_{Q_{j,k}\in\mathcal{F}_{v,l}}g^{\varepsilon}_{j,k}2^{nj/2}2^{j|\gamma|}\left(\frac{\partial^{\gamma}}{\partial
x^{\gamma}}\Phi^{\varepsilon}\right)(2^{j}x-k).
\end{eqnarray*}
Hence
\begin{eqnarray*}
S_{t-|\gamma|}(g_{u,\gamma})(x)&=&\left(\sum_{(\varepsilon,j,k)\in\Lambda_{n}}2^{j(n+2t-2|\gamma|)}
|g^{\varepsilon}_{j,k}|^{2}2^{j|\gamma|}\chi(2^{j}x-k)\right)^{1/2}\\
&=&\left(\sum_{(\varepsilon,j,k)\in\Lambda_{n}}2^{j(n+2t)}|g^{\varepsilon}_{j,k}|^{2}\chi(2^{j}x-k)\right)^{1/2}\\
&=&S_{t}(g_{u})(x).
\end{eqnarray*}
On the other hand, if $f\in M^{t,\tau}_{r,p}(\mathbb{R}^{n})$, then
$$\int_{Q}\left(\sum_{\varepsilon\in E_{n}, Q_{j,k}\subset Q}2^{j(n+2t)}|f^{\varepsilon}_{j,k}|^{2}\chi(2^{j}x-k)\right)^{p/2}dx
\leq C(1-\log_{2} |Q|)^{-p\tau}|Q|^{1-{p(r+t)}/{n}}$$ and
$(f_{\beta})^{\varepsilon}_{j,k}=2^{j|\beta|}f^{\varepsilon}_{j,k}$.
We have
\begin{eqnarray*}
&&\int_{Q}\left(\sum_{\varepsilon\in E_{n}, Q_{j,k}\subset
Q}2^{j(n+2t-2|\beta|)}|f^{\beta,\varepsilon}_{j,k}|^{2}\chi(2^{j}x-k)\right)^{p/2}dx\\
&=&\int_{Q}\left(\sum_{\varepsilon\in E_{n}, Q_{j,k}\subset
Q}2^{j(n+2t-2|\beta|)}2^{2j|\beta|}|f^{\varepsilon}_{j,k}|^{2}\chi(2^{j}x-k)\right)^{p/2}dx\\
&\leq&C(1-\log_{2}|Q|)^{-p\tau}|Q|^{1-{p(t-|\beta|+r+|\beta|)}/{n}},
\end{eqnarray*}
that is, $f_{\beta}(x)\in
M^{t-|\beta|,\tau}_{r+|\beta|,p}(\mathbb{R}^{n})$.

For any cube $Q$, the function
$f_{\beta}(x)=\sum\limits_{(\varepsilon,j,k)\in\Lambda_{n}}f^{\beta,\varepsilon}_{j,k}\Phi^{\varepsilon}_{j,k}(x)$
and any dyadic cube $Q\subset S_{u}$, we divide the product
$f_{\beta}\cdot g_{u,\gamma}$ into the following parts.
\begin{eqnarray*}
T^{\beta,\gamma}_{1}(x)&=&\sum_{(\varepsilon,j,k)\in\Lambda_{n},
j=j_{u,s}}~\sum_{|l|\leq\sqrt{n}2^{(M+2)n}}f^{\beta, 0}_{j,k+l}g^{\gamma,0}_{j,k}\Phi^{0}_{j,k+l}(x)\Phi^{0}_{j,k}(x);\\
T^{\beta,\gamma}_{2}(x)&=&\sum_{(\varepsilon,j,k)\in\Lambda_{n},
j\geq j_{u,s}}~\sum_{|l|\leq\sqrt{n}2^{(M+2)n}}f^{\beta,0}_{j,k+l}g^{\gamma,\varepsilon}_{j,k}\Phi^{0}_{j,k+l}(x)\Phi^{\varepsilon}_{j,k}(x);\\
T^{\beta,\gamma}_{3}(x)&=&\sum_{(\varepsilon,j,k)\in\Lambda_{n},
j\geq j_{u,s}}~\sum_{|l|\leq\sqrt{n}2^{(M+2)n}}f^{\beta,\varepsilon}_{j,k+l}g^{\gamma,0}_{j,k}\Phi^{\varepsilon}_{j,k+l}(x)\Phi^{0}_{j,k}(x);\\
T^{\beta,\gamma}_{4}(x)&=&\sum_{(\varepsilon,\varepsilon',j,k,
l)\in\Lambda'_{n},
j\geq j_{u,s}}~\sum_{|l|\leq\sqrt{n}2^{(M+2)n}}f^{\beta,\varepsilon'}_{j,k+l}g^{\gamma,\varepsilon}_{j,k}\Phi^{\varepsilon'}_{j,k+l}(x)\Phi^{\varepsilon}_{j,k}(x);\\
T^{\beta,\gamma}_{5}(x)&=&\sum_{(\varepsilon,j,k)\in\Lambda_{n},
j\geq
j_{u,s}}f^{\beta,\varepsilon}_{j,k}g^{\gamma,\varepsilon}_{j,k}\left((\Phi^{\varepsilon}_{j,k}(x))^{2}-2^{{nj}/{2}}
\Phi^{0}_{j,k}(x)\right);\\
T^{\beta,\gamma}_{6}(x)&=&\sum_{(\varepsilon,j,k)\in\Lambda_{n},
j\geq
j_{u,s}}f^{\beta,\varepsilon}_{j,k}g^{\gamma,\varepsilon}_{j,k}2^{{nj}/{2}}\Phi^{0}_{j,k}(x).
\end{eqnarray*}
Similar to the method used in the case $0\leq t<1$, we can complete
the proof of the case $t\geq 1$. This completes the proof of this
theorem.
\end{proof}
By Theorem \ref{th11}, we can get the following lemma.
\begin{lemma}\label{le20}
Given $r>0, t\geq 0, 1<p<{n}/{(r+t)}$ and $\tau>{1}/{p'}$. If
$f\in M^{t,\tau}_{r,p}(\mathbb{R}^{n})$ and $g_{u}$ is a
$(t+r,2^{u}, E_{u})-$combination atom, then
$$\|fg_{u}\|_{H^{t,p}}\leq C(1+u)^{-\tau}2^{u}|E_{u}|^{1/p}.$$
\end{lemma}
\begin{theorem}\label{th5}
Given $r>0, t\geq 0, 1<p<{n}/{(r+t)}$ and $\tau>{1}/{p'}$. If
$f\in M^{t,\tau}_{r,p}(\mathbb{R}^{n})$, then $f\in
X^{t}_{r,p}(\mathbb{R}^{n})$.
\end{theorem}
\begin{proof}
By Lemma \ref{le20}, we have
\begin{eqnarray*}
\|fg\|_{H^{t,p}}&\leq&\sum_{u\geq0}\|fg_{u}\|_{H^{t,p}}\\
&\leq&C\sum_{u\geq0}(1+u)^{-\tau}2^{u}|E_{u}|^{1/p}\\
&\leq&C\left(\sum_{u\geq0}(1+u)^{-p'\tau}\right)^{1/p'}\left(\sum_{u\geq0}2^{pu}|E_{u}|\right)^{1/p}\\
&\leq&C\|g\|_{H^{t+r,p}}.
\end{eqnarray*}
\end{proof}

\section{The sharpness for  the multiplier spaces
$M^{t, \tau }_{r,p}(\mathbb{R}^{n})$} In this section,  applying our
wavelet characterization of multiplier spaces and fractal theory, we
prove that the scope of the index of
$M^{t,\tau}_{r,p}(\mathbb{R}^{n})$ obtained in Theorem \ref{th5} is
sharp for $r+t<1$.  Precisely, by Meyer wavelets, we construct a
counterexample to show that  Theorem
\ref{th5} is not true for the case $0\leq\tau\leq {1}/{p'}$.

Our key idea  is to construct a group of sets $C_{s}$ composed by
special dyadic cubes and fractal set $H$ with Hausdorff dimension $n-(t+r)p$.
Denote by $S_{s}$ the union
$\bigcup\limits_{Q\in C_{s}} Q $ and $H=\bigcap\limits_{s\geq 1} S_{s}$.
By the above dyadic cubes $S_{s}, s\geq0$, we construct a special $L^{p}$ function
$g(x)$, which is bounded on $S_{s}\backslash S_{s+1}$ for all $s\geq
1$. The fractional integration $I_{r+t}g(x)$ bumps on the fractal
set $H$. Then we construct a multiplier $f(x)$ such that its wavelet
coefficients
 are based on these special dyadic cubes $C_{s}$ for all $s\geq 1$.
Applying our wavelet characterization of multiplier spaces,  we
prove that the product of the above multiplier $f(x)$ and the function
$I_{r+t}g(x)$ will go out the desired space $H^{t,p}(\mathbb{R}^{n})$.

For the above purpose, we  give first another characterization of
$X^{t}_{r,p}(\mathbb{R}^{n})$ associated with the fractional
integration. Let $\Psi(x)\geq0$, $\Psi(x)\in C^{\infty}_{0}(B(0,2))$
with $\Psi(x)=1$ on $B(0,1)$. We know that $\tilde{g}\in
H^{t+r,p}(\mathbb{R}^{n})$ if and only if there exists $g\in
L^{p}(\mathbb{R}^{n})$ such that $\tilde{g}(x)=J_{t+r}g(x)=\int
g(y)\Psi(|x-y|)|x-y|^{t+r-n}dy$. For $g\in L^{p}(\mathbb{R}^{n})$,
let $g_{j,k}=\int(2^{-j}+|y-2^{-j}k|)^{t+r-n}g(y)dy$. We define the
following function space.
\begin{definition}\label{def10}
Given $r>0$, $t\geq0$ and $1<p<{n}/{(r+t)}$. For
$f(x)=\sum\limits_{(\varepsilon,j,k)\in\Lambda_{n}}f^{\varepsilon}_{j,k}\Phi^{\varepsilon}_{j,k}(x)$,
$f\in Q^{t}_{r,p}(\mathbb{R}^{n})$ if and only if
$$\int\left(\sum_{(\varepsilon,j,k)\in\Lambda_{n}}2^{j(n+2t)}|g_{j,k}|^{2}|f^{\varepsilon}_{j,k}|^{2}\chi(2^{j}x-k)\right)^{p/2}dx\leq C,$$
where $g\in L^{p}(\mathbb{R}^{n})$ and $g(x)\geq0$.
\end{definition}
Let $f(x)=\sum\limits_{(\varepsilon,j,k)\in
\Lambda_{n}}f^{\varepsilon}_{j,k}\Phi^{\varepsilon}_{j,k}(x)$ and
$h(x)=\sum\limits_{(\varepsilon,j,k)\in
\Lambda_{n}}h^{\varepsilon}_{j,k}\Phi^{\varepsilon}_{j,k}(x)$.
Define
$$T_{\Phi}(f,h)(x)=\sum_{j\in\mathbb{N}}\sum_{k\in\mathbb{Z}^{n}}2^{jn}f^{\varepsilon}_{j,k}h^{\varepsilon}_{j,k}\Phi(2^{j}x-k).$$
Similar to Claim 2 in Theorem \ref{th3}, we can get
\begin{lemma}\label{le14}
Given $t\geq 0,r>0$ and $t+r<1<p<{n}/{(t+r)}$. Let $g\in
H^{t+r,p}(\mathbb{R}^{n})$ and $h\in H^{-t,p'}(\mathbb{R}^{n})$. The
function $f(x)\in S^{\Phi, t}_{r,p}(\mathbb{R}^{n})$ if and only if
$$|\left<T_{\Phi}(f,h),\ g\right>|\leq
C\|g\|_{H^{r+t,p}(\mathbb{R}^{n})}\|h\|_{H^{-t,p'}(\mathbb{R}^{n})}.$$
\end{lemma}
Now we give another characterization of
$X^{t}_{r,p}(\mathbb{R}^{n})$.
\begin{theorem}\label{th10}
Given $t\geq 0,r>0$ and $t+r<1<p<{n}/{(t+r)}$. $f\in
X^{t}_{r,p}(\mathbb{R}^{n})$ if and only if $f\in
M^{t}_{r,p}(\mathbb{R}^{n})$ and $f\in Q^{t}_{r,p}(\mathbb{R}^{n})$.
\end{theorem}
\begin{proof}
By modifying the coefficients such that
$f^{\varepsilon}_{j,k}h^{\varepsilon}_{j,k}=|f^{\varepsilon}_{j,k}||h^{\varepsilon}_{j,k}|$,
we could suppose $g(y)\geq 0$. By Theorem \ref{th3} (ii) and Lemma
\ref{le14}, we know that $f\in S^{\Phi, t}_{r,p}(\mathbb{R}^{n})$ is
equivalent to
\begin{eqnarray}\label{eq1}
&&\int\int\sum_{(\varepsilon,j,k)\in
\Lambda_{n}}|f^{\varepsilon}_{j,k}||h^{\varepsilon}_{j,k}|2^{jn}\Phi(2^{j}x-k)\Psi(|x-y|)\frac{g(y)}{|x-y|^{n-(t+r)}}dy dx\\
&\leq& C\|g\|_{L^{p}}\|h\|_{H^{-t,p'}}, \forall g\geq0.\nonumber
\end{eqnarray}
By a calculation of
the integral associated to $dx$, we get
$$\int2^{jn}\Phi(2^{j}x-k)\Psi(|x-y|)|x-y|^{t+r-n}dx\simeq (2^{-j}+|y-2^{-j}k|)^{(t+r-n)}.$$
So (\ref{eq1}) is equivalent to the following inequality:
\begin{equation}\label{eq2}
\int\sum_{(\varepsilon,j,k)\in
\Lambda_{n}}|f^{\varepsilon}_{j,k}||h^{\varepsilon}_{j,k}|(2^{-j}+|y-2^{-j}k|)^{t+r-n}g(y)dy\leq
C\|g\|_{L^{p}}\|h\|_{H^{-t,p'}}, \forall g\geq0.
\end{equation}
That is to say
\begin{eqnarray}\label{eq3}
&&\int\int\sum_{(\varepsilon,j,k)\in
\Lambda_{n}}2^{jn}|f^{\varepsilon}_{j,k}||h^{\varepsilon}_{j,k}|\chi(2^{j}x-k)\frac{g(y)}{(2^{-j}+|y-2^{-j}k|)^{n-(t+r)}}dx dy\\
&\leq& C\|g\|_{L^{p}}\|h\|_{H^{-t,p'}}, \forall g\geq0.\nonumber
\end{eqnarray}
Let
$S_{-t}h(x)=\left(\sum\limits_{(\varepsilon,j,k)\in\Lambda_{n}}2^{j(n-2t)}|h^{\varepsilon}_{j,k}|^{2}\chi(2^{j}x-k)\right)^{1/2}$.
We have $\|S_{-t}h\|_{p'}\simeq\|h\|_{H^{-t,p'}}$. Then the
inequality (\ref{eq3}) is equivalent to
\begin{eqnarray}\label{eq4}
&&\int\left(\sum_{(\varepsilon,j,k)\in
\Lambda_{n}}2^{j(n+2t)}|f^{\varepsilon}_{j,k}|^{2}|g_{j,k}|^{2}\chi(2^{j}x-k)\right)^{p/2}dx
\leq C\|g\|^{p}_{L^{p}}, \forall g\geq0.
\end{eqnarray}
The above inequality is equivalent to $f\in
Q^{t}_{r,p}(\mathbb{R}^{n})$.
\end{proof}
\begin{theorem}\label{th6}
If $0\leq\tau\leq\frac{1}{p'}$, there exists $f\in M^{t,
\tau}_{r,p}(\mathbb{R}^{n})$ such that $f$ does not belong to
$X^{t}_{r,p}(\mathbb{R}^{n})$.
\end{theorem}
\begin{proof}
We use Meyer wavelets and suppose that $\varepsilon=(1,1,\cdots,
1)$ and $\Phi^{\varepsilon}(0)>0$. First of all, we construct a
group of self similar cubes $\left\{Q_{u}\right\}$ such that the
limitation is a set with Hausdorff measure $n-(t+r)p$.

We construct two series of integers $\{\tau_{s}\}_{s\geq1}$ and
$\{v_{s}\}_{s\geq1}$ such that
$v_{1}=\max\left\{[\frac{4n}{n-(t+r)p}], 2^{(M+2)n}\right\}$, $1\leq
\tau_{s}<[\frac{4n}{n-(t+r)p}]\leq
v_{s}\leq\max\left\{[\frac{5n}{n-(t+r)p}], 2^{(M+2)n}\right\}$.
Denote $\sigma_{s}=\sum\limits_{1\leq i\leq s}\tau_{i}$ and
$u_{s}=\sum\limits_{1\leq i\leq s}v_{i}$. We take $\tau_{s}$ such
that there exists $C>0$ satisfying that
$2^{n\sigma_{s}-(n-(t+r)p)u_{s}}\geq 2^{-nu_{1}}$ and
$\lim\limits_{s\rightarrow+\infty}2^{n\sigma_{s}-(n-(t+r)p)u_{s}}=C$
 and
$n-(t+r)p=\lim\limits_{s\rightarrow+\infty}\frac{n\sigma_{s}}{u_{s}}.$
For $\tau_{s}$ and $v_{s}$, denote $l_{s}=(l_{s,1}, l_{s,2}, \cdots,
l_{s,n})\in\mathbb{Z}^{n}_{s,0}$, if $l_{s}\in\mathbb{Z}^{n}$ and
for $i=1,2,\cdots, n$, we have $0\leq l_{s,i}<2^{\tau_{s}-1}$ or
$0\leq l_{s,i}+2^{\tau_{s}-1}-2^{v_{s}}<2^{\tau_{s}-1}$. Denote
$\mathbb{Z}^{n}_{1}=\mathbb{Z}^{n}_{1,0}$ and for $s\geq 2$, denote
$k_{s}\in\mathbb{Z}^{n}_{s}$, if there exists
$k_{s-1}\in\mathbb{Z}^{n}_{s-1}$ and $l_{s}\in\mathbb{Z}^{n}_{s,0}$
such that $k_{s}=2^{v_{s}}k_{s-1}+l_{s}$.

We divide the unit dyadic cube $Q_{0}=[0,1]^{n}$ into $2^{nv_{1}}$
dyadic cubes. Then we reserve the $2^{n\tau_{1}}$ dyadic cubes which
are near the $2^{n}$ vertex points, that is, we reserve
$C_{1}=\left\{Q_{v_{1}, l_{1}}:\
l_{1}\in\mathbb{Z}^{n}_{1,0}\right\}$ and denote $x\in S_{1}$, if
there exists $l_{1}\in \mathbb{Z}^{n}_{1,0}$ such that $x\in
Q_{v_{1}, l_{1}}$.

For the dyadic cube $Q_{v_{1}, l_{1}}\in C_{1}$, we divide it into
$2^{nv_{2}}$ dyadic cubes, we reserve the $2^{n\tau_{2}}$ dyadic
cubes which are near the $2^{n}$ vertex points, that is, we reserve
$C_{2,l_{1}}=\left\{Q_{u_{2}, 2^{v_{2}}l_{1}+l_{2}},
l_{2}\in\mathbb{Z}^{n}_{1,0}\right\}$. For $s=2$, denote
$C_{2}=\left\{Q_{u_{2}, k_{2}}, k_{2}=2^{v_{2}}l_{1}+l_{2}, l_{1},
l_{2}\in\mathbb{Z}^{n}_{1,0}\right\}$ and denote $x\in S_{2}$, if
there exists $k_{2}\in\mathbb{Z}^{n}_{2}$ such that $x\in Q_{u_{2},
k_{2}}$.

We continue this process until infinity, we get a series of dyadic
cubes $C_{s}$ and sets $ S_{s} $. We know that
$|S_{s}|=2^{n(\sigma_{s}-u_{s})}$ and the limitation of
$\left\{S_{s}\right\}$ is a fractal set $H$ with Hausdorff dimension
$n-(t+r)p$.

For $s\geq1$, let $f_{s}(x)=\sum\limits_{Q_{u_{s}, k_{s}}\subset
C_{s}} f^{\varepsilon}_{u_{s}, k_{s}}\Phi^{\varepsilon}_{u_{s},
k_{s}}(x)$, where $f^{\varepsilon}_{u_{s},
k_{s}}=s^{-{1}/{p'}}2^{-tu_{s}}2^{(t+r)u_{s}}$. Let
$f(x)=\sum\limits_{s\geq 1} f_{s}(x)$. Applying the wavelet
characterization of Morrey spaces, $f\in
M^{t,{1}/{p'}}_{r,p}(\mathbb{R}^{n})$. In fact, for $s, l\geq1$
and any cube $Q$ with $2^{-nu_{s+1}}\leq|Q|<2^{-nu_{s}}$, we have
$|Q\cap(S_{s+l}\backslash S_{s+l+1})|\leq
C2^{-n(\sigma_{s+l}-\sigma_{s})-nu_{s}}$. Hence we get
\begin{eqnarray*}
&&\int_{Q}\left(\sum_{Q_{u_{s}, k_{s}}\subset
Q}2^{u_{s}(n+2t)}|f^{\varepsilon}_{u_{s},
k_{s}}|^{2}\chi(2^{u_{s}}x-k_{s})\right)^{p/2}dx\\
&\lesssim&\int_{Q\backslash S_{s+1}}\left(\sum_{Q_{u_{s},
k_{s}}\subset
Q}2^{u_{s}(n+2t)}s^{-{2}/{p'}}2^{-2tu_{s}}2^{2(r+t)u_{s}}\chi(2^{u_{s}}x-k_{s})\right)^{p/2}dx\\
&+&\sum_{l\geq1}\int_{Q\cap (S_{s+l}\backslash
S_{S+l+1})}\left(\sum_{Q_{u_{s}, k_{s}}\subset
Q}2^{u_{s}(n+2t)}s^{-{2}/{p'}}2^{-2tu_{s}}2^{2(r+t)u_{s}}\chi(2^{u_{s}}x-k_{s})\right)^{p/2}dx\\
&\lesssim&s^{-{p}/{p'}}2^{(t+r)pu_{s}}2^{-nu_{s}}+C\sum_{l\geq1}(s+l)^{-{p}/{p'}}2^{(t+r)pu_{s+l}}2^{-n(\sigma_{s+l}-\sigma_{s})-nu_{s}}\\
&\lesssim&(1-\log_{2}|Q|)^{-{p}/{p'}}|Q|^{1-{(t+r)p}/{n}}.
\end{eqnarray*}
Let $\delta$ be a sufficient small positive real number. For
$\forall x\in [0,1]^{n}\backslash S_{1}$, then $g(x)=1$. For $s\geq
1$ and $x\in S_{s}\backslash S_{s+1}$, we take
$g(x)=s^{-{1}/{p}}[\log_{2}(1+s)]^{-{(1+\delta)}/{p}}2^{(t+r)u_{s}}$.
Then we have $g\in L^{p}([0,1]^{n})$. In fact,
\begin{eqnarray*}
\int_{[0, 1]^{n}}|g(x)|^{p}dx&=&\int_{[0,1]^{n}\backslash
S_{1}}g^{p}(x)dx+\sum_{s\geq1}\int_{S_{s}\backslash
S_{s+1}}g^{p}(x)dx\\
&\leq&C+C\sum_{s\geq1}\int_{S_{s}\backslash
S_{s+1}}s^{-1}[\log_{2}(1+s)]^{-(1+\delta)}2^{(t+r)pu_{s}}dx\\
&\leq&C\sum_{s\geq1}s^{-1}[\log_{2}(1+s)]^{-(1+\delta)}2^{(t+r)pu_{s}}2^{n(\sigma_{s}-u_{s})}\\
&\leq&C.
\end{eqnarray*}
Now we estimate the coefficients $|g_{u_{s},k_{s}}|$. We divide the
estimate into two cases.

(1) For $\text{ dist }(Q_{u_{s},k_{s}},S_{s})\leq 2^{-u_{s}}$,
\begin{eqnarray*}
g_{j,k}&=&\int_{S_{s}}g(y)(2^{-u_{s}}+|y-2^{-u_{s}}k_{s}|)^{(t+r-n)}dy\\
&+&\sum_{1\leq l\leq s-1}\int_{S_{l}\backslash
S_{l+1}}l^{-{1}/{p}}[\log_{2}(1+l)]^{-{(1+\delta)}/{p}}2^{{n(u_{l}-\sigma_{l})}/{p}}(2^{-u_{l}}+|y-2^{-u_{l}}k_{l}|)^{(t+r-n)}dy\\
&\geq&s^{-{1}/{p}}[\log_{2}(1+s)]^{-{(1+\delta)}/{p}}+\sum_{1\leq
l\leq s-1}s^{-{1}/{p}}[\log_{2}(1+s)]^{-{(1+\delta)}/{p}}\\
&\geq&s^{{1}/{p'}}[\log_{2}(1+s)]^{-{(1+\delta)}/{p}}.
\end{eqnarray*}

(2) For  $\text{ dist }(Q_{u_{s},k_{s}},S_{s})>2^{-u_{s}}$, there
exists $l<s$ such that $Q_{u_{s},k_{s}}\subset S_{l-1}\backslash
S_{l}$. It is easy to see that $g_{j,k}$ is equivalent to
$l^{-{1}/{p'}}[\log_{2}(1+l)]^{-{(1+\delta)}/{p}}$.

Finally, we have
\begin{eqnarray*}
&&\int\left(\sum_{(\varepsilon,j,k)\in\Lambda_{n}}2^{j(n+2t)}|g_{j,k}|^{2}|f^{\varepsilon}_{j,k}|^{2}\chi(2^{j}x-k)\right)^{{p}/{2}}dx\\
&\geq&C\sum_{s\geq2}\left(\sum_{\text{ dist }(Q_{u_{s},k_{s}},S_{s})\leq 2^{-u_{s}}}
2^{u_{s}(n+2t)}|g_{u_{s},k_{s}}|^{2}|f^{\varepsilon}_{u_{s},k_{s}}|^{2}\chi(2^{u_{s}}x-k_{s})\right)^{{p}/{2}}dx\\
&\geq&C\sum_{s\geq2}\left(\sum_{\text{ dist
}(Q_{u_{s},k_{s}},S_{s})\leq
2^{-u_{s}}}2^{u_{s}(n+2t)}s^{{2}/{p'}}[\log_{2}(1+s)]^{-{2(1+\delta)}/{p}}
s^{-{2}/{p'}}2^{-2tu_{s}}2^{2(t+r)u_{s}}\chi(2^{u_{s}}x-k_{s})\right)^{{p}/{2}}dx\\
&\geq&\sum_{s\geq2}[\log_{2}(1+s)]^{-(1+\delta)}2^{(t+r)pu_{s}}2^{n(\sigma_{s}-u_{s})}\\
&\geq&\sum_{s\geq2}[\log_{2}(1+s)]^{-(1+\delta)}=\infty.
\end{eqnarray*}
This completes the proof.
\end{proof}

\section{An application to Schr\"{o}dinger type operators with non-smooth potentials}
In \cite{MV},  the multipliers
from $H^{1,2}(\mathbb{R}^{n})$ to
 $H^{-1,2}(\mathbb{R}^{n})$ were studied by V. Maz'ya and I. E. Verbitsky. For a Schr\"odinger operator $L=I-\Delta+V$,
they established many sufficient and necessary conditions such that
$V$ is a multiplier from ${H}^{1,2}(\mathbb{R}^{n})$ to
${H}^{-1,2}(\mathbb{R}^{n})$. In this section, we give an
application of the wavelet characterization of
$X^{t}_{r,p}(\mathbb{R}^{n})$ to the Schr\"{o}dinger type operator
$I+(-\Delta)^{{r}/{2}}+V$.

For $V\in
M^{t,\tau}_{r,p}(\mathbb{R}^{n}),\ (\tau>{1}/{p'})$ and $g(x)\in
H^{t,p}(\mathbb{R}^{n})$, we want to find a solution $f\in H^{t+r,p}(\mathbb{R}^{n})$ to the equation
\begin{equation}\label{eq5}
 (I+(-\Delta)^{{r}/{2}}+V) f(x) = g(x).
 \end{equation}

\begin{remark}
Fixed $r>0$, $t\geq 0$ and $1<p<{n}/{(r+t)}$.
\begin{itemize}
\item[(i)] If there exists a $\delta>0$ such that
$\|V\|_{C^{r+t+\delta}}$ is sufficient small,
according to the continuity of Calder\'on-Zygmund operator
$(I+(-\Delta)^{{r}/{2}}) ( I+(-\Delta)^{{r}/{2}}+V)^{-1}$,
the equation (\ref{eq5}) can be solved easily.
But if we consider a non smooth potential
$V\in M^{t,\tau}_{r,p}(\mathbb{R}^{n})$,
applying the same proof in Lemmas \ref{le4} and \ref{le5},
it is possible that $V$ is not a $L^{\infty}$ function.
\item[(ii)] The condition $\tau>1/p'$ can not be weaken to $\tau\geq 1/p'$.
In fact, according to our counterexample in \S 5, if $r+t<1$, there
exists some $V\in M^{t,1/p'}_{r,p}(\mathbb{R}^{n})$ such that the
operator $I+(-\Delta)^{{r}/{2}}+V$ is not continuous from
$H^{t+r,p}(\mathbb{R}^{n})$ to $H^{t,p}(\mathbb{R}^{n})$.
\end{itemize}
\end{remark}

Now, we use our sufficient condition of multiplier spaces
$X^{t}_{r,p}(\mathbb{R}^{n})$
to get the solution of the equation (6.1).
We need the following two operators.
For $t,r>0$, let $T_{t,r}=[I+(-\Delta)^{t/2}][I+(-\Delta)^{r/2}]
$ and $S_{t,r}=[I+(-\Delta)^{t/2}]VT_{t,r}^{-1}$. In the following lemma, we prove that the operator
$S_{t,r}=[I+(-\Delta)^{t/2}]VT_{t,r}^{-1}$ is bounded on $L^{p}(\mathbb{R}^{n})$

\begin{lemma}\label{le15}
 Given $r>0, t\geq 0, 1<p<{n}/{(r+t)}$ and $\tau>{1}/{p'}$. If $V(x)\in M^{t,\tau}_{r,p}(\mathbb{R}^{n})$,
  the operator $S_{t,r}$ is bounded from $L^{p}(\mathbb{R}^{n})$ to $L^{p}(\mathbb{R}^{n})$ with the operator
  norm less than $C_{t, r, \tau, p}\|V\|_{ M^{t,\tau}_{r,p}}$, where $C_{t, r, \tau, p}$ denotes a constant associated with $t, r, \tau, p$.
\end{lemma}
\begin{proof}
By Theorem \ref{th5}, for any $f\in L^{p}(\mathbb{R}^{n})$, we have
\begin{eqnarray*}
\|S_{t,r}f\|_{L^{p}} \leq &\|[I+(-\Delta)^{t/2}]VT^{-1}_{t,r}f\|_{L^{p}} &\leq \|VT^{-1}_{t,r}f\|_{H^{t,p}}\\
\leq & C_{r,t,\tau,
p}\|V\|_{M^{t,\tau}_{r,p}}\|T^{-1}_{t,r}f\|_{H^{t+r,p}} &\leq
C_{r,t,\tau, p}\|V\|_{M^{t,\tau}_{r,p}}\|f\|_{L^{p}}.
\end{eqnarray*}
Then  $S_{t,r}$ is a bounded operator on $L^{p}(\mathbb{R}^{n})$
with the norm less than $C_{r,t,\tau, p}\|V\|_{M^{t,\tau}_{r,p}}$.
\end{proof}

In the following lemma, we prove that $(I+S_{t,r})$ is invertible in
$L^{p}(\mathbb{R}^{n})$ and the inverse $(I+S_{t,r})^{-1}$ can be
written formally as $\sum\limits_{n=0}^{\infty}(-1)^{n}S^{n}_{t,r}$.
\begin{lemma}\label{le16}
Given $r>0, t\geq 0, 1<p<{n}/{(r+t)}$ and $\tau>{1}/{p'}$. If\
$\|V(x)\|_{ M^{t,\tau}_{r,p}}<{1}/{C_{r,t,\tau, p}}$, the
operator $I+S_{t,r}$ is invertible in $L^{p}(\mathbb{R}^{n})$.
\end{lemma}
\begin{proof}
By Lemma \ref{le15}, the operator $S_{t,r}$ is bounded on $L^{p}(\mathbb{R}^{n})$.
Hence for any $f\in L^{p}(\mathbb{R}^{n})$,
\begin{eqnarray*}
\left\|\sum^{\infty}_{n=0}(-1)^{n}S^{n}_{t,r}f\right\|_{L^{p}}&\leq&\sum^{\infty}_{n=0}\left\|S^{n}_{t,r}f\right\|_{L^{p}}
\leq\sum^{\infty}_{n=0}(C_{r,t,\tau,
p}\|V\|_{M^{t,\tau}_{r,p}})^{n}\|f\|_{L^{p}}.
\end{eqnarray*}
If $\|V\|_{M^{t,\tau}_{r,p}}<{1}/{C_{r,t,\tau, p}}$, the above
series is convergent in $L^{p}(\mathbb{R}^{n})$. Further,
\begin{eqnarray*}
(I+S_{t,r})\left[\sum^{\infty}_{n=0}(-1)^{n}S^{n}_{t,r}\right]&=&\sum^{\infty}_{n=0}(-1)^{n}S^{n}_{t,r}
-\sum^{\infty}_{n=1}(-1)^{n}S^{n}_{t,r} =I.
\end{eqnarray*}
Similarly, we can also get
$\left[\sum^{\infty}\limits_{n=0}(-1)^{n}S^{n}_{t,r}\right](I+S_{t,r})=I$,
 that is, the operator
$I+S_{t,r}$ is invertible in $L^{p}(\mathbb{R}^{n})$.
\end{proof}

\begin{theorem}
Given $r>0, t\geq 0, 1<p<{n}/{(r+t)}$ and $\tau>{1}/{p'}$. If
$\|V\|_{ M^{t,\tau}_{r,p}}<{1}/{C_{r,t,\tau, p}}$, then for
$g(x)\in H^{t,p}(\mathbb{R}^{n})$, there exists a unique solution
$f(x)\in H^{r+t,p}(\mathbb{R}^{n})$ for equation (6.1).
\end{theorem}
\begin{proof}
Because $g\in H^{t,p}(\mathbb{R}^{n})$, we have $\tilde{g}=(I+(-\Delta)^{t/2})g(x)\in L^{p}(\mathbb{R}^{n})$.
By Lemma \ref{le16}, the operator $(I+S_{t,r})$ is invertible in $L^{p}(\mathbb{R}^{n})$.
Hence we can get there exists a unique solution to the following equation in $L^{p}(\mathbb{R}^{n})$,
\begin{equation}\label{eq6}
(I+S_{t,r})\tilde{f}=\tilde{g},
\end{equation}
where $\tilde{g}\in L^{p}(\mathbb{R}^{n})$. Hence for above $\tilde{g}\in L^{p}(\mathbb{R}^{n})$,
$$\tilde{f}=(I+S_{t,r})^{-1}\tilde{g}=(I+S_{t,r})^{-1}(I+(-\Delta)^{{t}/{2}})g(x)\in L^{p}(\mathbb{R}^{n})$$
is a solution to the equation (\ref{eq6}). Write $ f= (I+(-\Delta)^{{r}/{2}})^{-1}(I+(-\Delta)^{{t}/{2}})^{-1} \tilde{f}$.
Then $\tilde{f}(x)\in L^{p}(\mathbb{R}^{n})$ is equivalent to $f\in H^{r+t,p}(\mathbb{R}^{n})$. It is easy to
verify that $f$ is a solution to the equation (\ref{eq5}). This completes the proof.
\end{proof}

\end{document}